\documentclass[11pt,a4paper,onesided]{article}

\usepackage{ amssymb, amsmath, amsthm, graphics, xspace, enumerate}
\usepackage{graphicx}
\usepackage[a4paper,colorlinks=true,citecolor=blue,urlcolor=blue,linkcolor=blue,bookmarksopen=true,unicode=true,pdffitwindow=true]{hyperref}
 \hypersetup{pdfauthor={Nikola Sandri\'{c}}}
\hypersetup{pdftitle={Long-time behavior for a class of Feller
processes}}

\newtheorem{tm}{tm}[section]
\newtheorem{theorem}[tm]{Theorem}
\newtheorem{lemma}[tm]{Lemma}
\newtheorem{corollary}[tm]{Corollary}
\newtheorem{proposition}[tm]{Proposition}

\newtheorem{definition}[tm]{Definition}

\newcommand{\process}[1]{\{#1_t\}_{t\geq0}}
\newcommand{\chain}[1]{\{#1_n\}_{n\geq0}}
\newcommand {\R} {\ensuremath{\mathbb{R}}}

\newcommand {\N} {\ensuremath{\mathbb{N}}}
\newcommand {\CC} {\ensuremath{\mathbb{C}}}
\numberwithin{equation}{section}
\def\be{\begin{equation}}
\def\ee{\end{equation}}
\begin{document}

 \title{Long-time behavior for a class of Feller processes}
 \author{Nikola Sandri\'{c}\\ Department of Mathematics\\
         Faculty of Civil Engineering, University of Zagreb, Zagreb,
         Croatia \\
        Email: nsandric@grad.hr }

 \maketitle
\begin{center}
{
\medskip

} \end{center}

\begin{abstract}
In this paper, as a main result,  we derive a Chung-Fuchs type
condition for the recurrence   of   Feller processes associated with
pseudo-differential operators. In the L\'evy process case, this
condition  reduces to the classical and well-known Chung-Fuchs
condition. Further, we also discuss the recurrence and transience of
Feller processes with respect to the dimension of the state space
and Pruitt indices and  the recurrence and transience of
Feller-Dynkin diffusions and stable-like processes. Finally, in the
one-dimensional symmetric case, we study perturbations of Feller
processes which do not affect their recurrence and transience
properties, and we derive sufficient conditions for their recurrence
and transience in terms of the corresponding L\'evy measure. In
addition, some comparison conditions for recurrence and transience
also in terms of the L\'evy measures are obtained.

\end{abstract}
 \smallskip


%
%
% ------------------------------ INTRODUCTION ---------------------------------------
%
%

%\thispagestyle{empty}

\section{Introduction}\label{s1}
\ \ \ \

In this paper, we study the recurrence and transience property of
Feller processes associated with  pseudo-differential operators in
terms of the symbol. To be more precise, let
$(\Omega,\mathcal{F},\{\mathbb{P}^{x}\}_{x\in\R^{d}},$
$\process{\mathcal{F}},\process{\theta},\process{M})$, $\process{M}$
in the sequel, be a Markov process with  state space
$(\R^{d},\mathcal{B}(\R^{d}))$, where $d\geq1$ and
$\mathcal{B}(\R^{d})$ denotes the Borel $\sigma$-algebra on
$\R^{d}$. A family of linear operators $\process{P}$ on
$B_b(\R^{d})$ (the space of bounded and Borel measurable functions),
defined by $$P_tf(x):= \mathbb{E}^{x}[f(M_t)],\quad t\geq0,\
x\in\R^{d},\ f\in B_b(\R^{d}),$$ is associated with the process
$\process{M}$. Since $\process{M}$ is a Markov process, the family
$\process{P}$ forms a \emph{semigroup} of linear operators on the
Banach space $(B_b(\R^{d}),||\cdot||_\infty)$, that is, $P_s\circ
P_t=P_{s+t}$ and $P_0=I$ for all $s,t\geq0$. Here,
$||\cdot||_\infty$ denotes the supremum norm on the space
$B_b(\R^{d})$. Moreover, the semigroup $\process{P}$ is
\emph{contractive}, that is, $||P_tf||_{\infty}\leq||f||_{\infty}$
for all $t\geq0$ and all $f\in B_b(\R^{d})$, and \emph{positivity
preserving}, that is, $P_tf\geq 0$ for all $t\geq0$ and all $f\in
B_b(\R^{d})$ satisfying $f\geq0$. The \emph{infinitesimal generator}
$(\mathcal{A},\mathcal{D}_{\mathcal{A}})$ of the semigroup
$\process{P}$ (or of the process $\process{M}$) is a linear operator
$\mathcal{A}:\mathcal{D}_{\mathcal{A}}\longrightarrow B_b(\R^{d})$
defined by
$$\mathcal{A}f:=
  \lim_{t\longrightarrow0}\frac{P_tf-f}{t},\quad f\in\mathcal{D}_{\mathcal{A}}:=\left\{f\in B_b(\R^{d}):
\lim_{t\longrightarrow0}\frac{P_t f-f}{t} \ \textrm{exists in}\
||\cdot||_\infty\right\}.$$

A Markov process $\process{M}$ is said to be a \emph{Feller process}
if its corresponding  semigroup $\process{P}$ forms a \emph{Feller
semigroup}. This means that the family $\process{P}$ is a semigroup
of linear operators on the Banach space
$(C_\infty(\R^{d}),||\cdot||_{\infty})$  and it is \emph{strongly
continuous}, that is,
  $$\lim_{t\longrightarrow0}||P_tf-f||_{\infty}=0,\quad f\in
  C_\infty(\R^{d}).$$ Here, $C_\infty(\R^{d})$ denotes
the space of continuous functions vanishing at infinity. Let us
remark that every Feller process possesses the strong Markov
property and has c\`adl\`ag paths (see  \cite[Theorems 3.4.19 and
3.5.14]{jacobIII}). In the case of Feller processes, we call
$(\mathcal{A},\mathcal{D}_{\mathcal{A}})$ the \emph{Feller
generator} for short. Note that, in this case,
$\mathcal{D}_{\mathcal{A}}\subseteq C_\infty(\R^{d})$ and
$\mathcal{A}(\mathcal{D}_{\mathcal{A}})\subseteq C_\infty(\R^{d})$.
Further, if the set of smooth functions with compact support
$C_c^{\infty}(\R^{d})$ is contained in $\mathcal{D}_{\mathcal{A}}$,
that is, if the Feller generator
$(\mathcal{A},\mathcal{D}_{\mathcal{A}})$ of the Feller process
$\process{M}$ satisfies
    \begin{description}
      \item[(C1)]
      $C_c^{\infty}(\R^{d})\subseteq\mathcal{D}_{\mathcal{A}}$,
    \end{description}
 then, according to \cite[Theorem 3.4]{courrege-symbol},
$\mathcal{A}|_{C_c^{\infty}(\R^{d})}$ is a \emph{pseudo-differential
operator}, that is, it can be written in the form
\begin{align*}\mathcal{A}|_{C_c^{\infty}(\R^{d})}f(x) = -\int_{\R^{d}}q(x,\xi)e^{i\langle \xi,x\rangle}
\mathcal{F}(f)(\xi) d\xi,\end{align*}  where $\mathcal{F}(f)(\xi):=
(2\pi)^{-d} \int_{\R^{d}} e^{-i\langle\xi,x\rangle} f(x) dx$ denotes
the Fourier transform of the function $f(x)$. The function $q :
\R^{d}\times \R^{d}\longrightarrow \CC$ is called  the \emph{symbol}
of the pseudo-differential operator. It is measurable and locally
bounded in $(x,\xi)$ and continuous and negative definite as a
function of $\xi$. Hence, by \cite[Theorem 3.7.7]{jacobI}, the
function $\xi\longmapsto q(x,\xi)$ has for each $x\in\R^{d}$ the
following L\'{e}vy-Khintchine representation $$q(x,\xi) =a(x)-
i\langle \xi,b(x)\rangle + \frac{1}{2}\langle\xi,c(x)\xi\rangle -
\int_{\R^{d}}\left(e^{i\langle\xi,y\rangle}-1-i\langle\xi,y\rangle1_{\{z:|z|\leq1\}}(y)\right)\nu(x,dy),$$
where $a(x)$ is a nonnegative Borel measurable function, $b(x)$ is
an $\R^{d}$-valued Borel measurable function,
$c(x):=(c_{ij}(x))_{1\leq i,j\leq d}$ is a symmetric nonnegative
definite $d\times d$ matrix-valued Borel measurable function
 and $\nu(x,dy)$ is a Borel kernel on $\R^{d}\times
\mathcal{B}(\R^{d})$, called the \emph{L\'evy measure}, satisfying
$$\nu(x,\{0\})=0\quad \textrm{and} \quad \int_{\R^{d}}\min\{1,
|y|^{2}\}\nu(x,dy)<\infty,\quad x\in\R^{d}.$$ The quadruple
$(a(x),b(x),c(x),\nu(x,dy))$ is called the \emph{L\'{e}vy-quadruple}
of the pseudo-differential operator
$\mathcal{A}|_{C_c^{\infty}(\R^{d})}$ (or of the symbol $q(x,\xi)$).
In the sequel, we assume the following conditions on the symbol
$q(x,\xi)$:
\begin{description}
  \item[(C2)] $||q(\cdot,\xi)||_\infty\leq c(1+|\xi|^{2})$ for some
  $c\geq0$ and all $\xi\in\R^{d}$
  \item[(C3)] $q(x,0)=a(x)=0$ for all $x\in\R^{d}.$
\end{description}
Let us remark that, according to \cite[Lemma 2.1]{rene-holder},
condition (\textbf{C2}) is equivalent with the boundedness of the
coefficients of the symbol $q(x,\xi)$, that is,
$$||a||_\infty+||b||_\infty+||c||_{\infty}+\left|\left|\int_{\R^{d}}\min\{1,y^{2}\}\nu(\cdot,dy)\right|\right|_{\infty}<\infty,$$
and, according to \cite[Theorem 5.2]{rene-conserv}, condition
(\textbf{C3}) (together with condition (\textbf{C2})) is equivalent
with the conservativeness property of the process $\process{M}$,
that is, $\mathbb{P}^{x}(M_t\in\R^{d})=1$ for all $t\geq0$ and all
$x\in\R^{d}$. In the case when the symbol $q(x,\xi)$ does not depend
on the variable $x\in\R^{d}$, $\process{M}$ becomes a \emph{L\'evy
process}, that is, a stochastic process   with stationary and
independent increments and c\`adl\`ag paths. Moreover, unlike Feller
processes, every L\'evy process is uniquely and completely
characterized through its corresponding symbol (see \cite[Theorems
7.10 and 8.1]{sato-book}). According to this, it is not hard to
check that every L\'evy process satisfies conditions
(\textbf{C1})-(\textbf{C3}) (see \cite[Theorem 31.5]{sato-book}).
Thus, the class of processes we consider in this paper contains the
class of L\'evy processes.

In this paper, our main aim is to investigate the recurrence and
transience property of Feller processes  satisfying conditions
(\textbf{C1})-(\textbf{C3}). Except for L\'evy processes, whose
recurrence and transience property has been  studied extensively in
\cite{sato-book}, a few special cases of this problem have been
considered in the literature. More precisely, in
\cite{bjoern-overshoot}, \cite{franke-periodic},
\cite{franke-periodicerata}, \cite{sandric-periodic},
\cite{sandric-ergodic}, \cite{sandric-spa} and
\cite{sandric-rectrans}, by using different techniques (an overshoot
approach, characteristics of semimartingale approach and an approach
through the Foster-Lypunov drift criteria), the authors have
considered the recurrence and transience  of one-dimensional Feller
processes determined by a symbol of the form
$q(x,\xi)=\gamma(x)|\xi|^{\alpha(x)}$ (stable-like processes), where
$\alpha:\R\longrightarrow(0,2)$ and
$\gamma:\R\longrightarrow(0,\infty)$ (see Section \ref{s2} for the
exact definition of these processes). Further, by using the
Foster-Lyapunov drift criteria (see \cite{meyn-tweedie-III} or
\cite{meyn-tweedie-book}), in \cite{wang-ergodic} the author has
derived sufficient conditions for
 recurrence of one-dimensional Feller processes in terms of their
L\'evy quadruples. Finally, in
 \cite{rene-wang-feller}, by analyzing the characteristic function of  Feller processes, the authors have derived a  Chung-Fuchs type condition for
  transience (see Theorem \ref{tm1.3} for details). In this
 paper, our goal is to extend the above mentioned results in several
 different aspects as well as to answer  some natural questions regarding the recurrence and transience in order to  better understand  the long-time behavior of Feller processes. To be more precise, our main goal is to derive a Chung-Fuchs type
condition for the recurrence of a Feller process. Furthermore, we
study recurrence and transience in relation to the dimension of the
state space and Pruitt indices and recurrence and transience of
Feller-Dynkin diffusions and stable-like processes. Finally, we
study perturbations of symbols which will not affect the recurrence
and transience of the underlying Feller processes and we derive
sufficient conditions for the recurrence and transience in terms of
the underlying L\'evy measure and some comparison conditions for
recurrence and transience also in terms of their L\'evy measures.

Before stating the main results of this paper, we recall  relevant
definitions of the recurrence and transience of Markov processes in
the sense of S. P. Meyn and R. L. Tweedie (see
\cite{meyn-tweedie-book} or \cite{tweedie-mproc}).
\begin{definition}
Let $\process{M}$ be a strong Markov process with c\`adl\`ag paths
on the state space $(\R^{d},\mathcal{B}(\R^{d}))$, $d\geq1$. The
process $\process{M}$ is called
\begin{enumerate}
  \item [(i)] \emph{irreducible} if there exists a $\sigma$-finite measure $\varphi(dy)$ on
$\mathcal{B}(\R^{d})$ such that whenever $\varphi(B)>0$ we have
$\int_0^{\infty}\mathbb{P}^{x}(M_t\in B)dt>0$ for all $x\in\R^{d}$.
  \item [(ii)] \emph{recurrent} if it is
                      $\varphi$-irreducible and if $\varphi(B)>0$ implies $\int_{0}^{\infty}\mathbb{P}^{x}(M_t\in B)dt=\infty$ for all
                      $x\in\R^{d}$.
\item [(iii)] \emph{Harris recurrent} if it is $\varphi$-irreducible and if $\varphi(B)>0$ implies $\mathbb{P}^{x}(\tau_B<\infty)=1$ for all
                      $x\in\R^{d}$, where $\tau_B:=\inf\{t\geq0:M_t\in
                      B\}.$
 \item [(iv)] \emph{transient} if it is $\varphi$-irreducible
                       and if there exists a countable
                      covering of $\R^{d}$ with  sets
$\{B_j\}_{j\in\N}\subseteq\mathcal{B}(\R^{d})$, such that for each
$j\in\N$ there is a finite constant $c_j\geq0$ such that
$\int_0^{\infty}\mathbb{P}^{x}(M_t\in B_j)dt\leq c_j$ holds for all
$x\in\R^{d}$.
\end{enumerate}
\end{definition}
Let us remark that if $\{M_t\}_{t\geq0}$ is a $\varphi$-irreducible
Markov process, then the irreducibility measure $\varphi(dy)$ can be
maximized, that is, there exists a unique ``maximal" irreducibility
measure $\psi(dy)$ such that for any measure $\bar{\varphi}(dy)$,
$\{M_t\}_{t\geq0}$ is $\bar{\varphi}$-irreducible if, and only if,
$\bar{\varphi}\ll\psi$ (see \cite[Theorem 2.1]{tweedie-mproc}).
According to this, from now on, when we refer to irreducibility
measure we actually refer to the maximal irreducibility measure. In
the sequel, we consider  only the so-called open set irreducible
Markov processes, that is,  we consider only $\psi$-irreducible
Markov processes whose maximal irreducibility measure $\psi(dy)$
satisfies the following \emph{open set irreducibility} condition:
\begin{description}
  \item[(C4)] $\psi(O)>0$ for every open set $O\subseteq\R^{d}$.
\end{description}
Obviously, the Lebesgue measure $\lambda(dy)$ satisfies condition
(\textbf{C4}) and a Markov process $\process{M}$ will be
$\lambda$-irreducible if $\mathbb{P}^{x}(M_t\in B)>0$ for all $t>0$
and all $x\in\R^{d}$ whenever $\lambda(B)>0.$ In particular, the
process $\process{M}$ will be $\lambda$-irreducible if the
transition kernel $\mathbb{P}^{x}(M_t\in dy)$ possesses a density
function $p(t,x,y)$ such that $p(t,x,y)>0$ for all $t>0$ and all
$x,y\in\R^{d}.$ If $\process{M}$ is a Feller process satisfying
conditions (\textbf{C1})-(\textbf{C3}) and, additionally, the
following sector condition
\be\label{eq:1.1}\sup_{x\in\R^{d}}|\rm{Im}\,\it{q}(x,\xi)|\leq
  c\inf_{x\in\R^{d}}\rm{Re}\,\it{q}(x,\xi)\ee  for some $0\leq
  c<1$ and all $\xi\in\R^{d}$,
then a sufficient condition for the existence of a density function
$p(t,x,y)$, in terms of  the symbol $q(x,\xi)$, has been given in
\cite[Theorem 1.1]{rene-wang-feller} as the Hartman-Wintner
condition
\be\label{eq:1.2}\lim_{|\xi|\longrightarrow\infty}\frac{\inf_{x\in\R^{d}}\rm{Re}\,\it{q(x,\xi)}}{\log(1+|\xi|)}=\infty\ee
(see also Theorem \ref{tm2.6}). Let us remark that the sector
condition in (\ref{eq:1.1}) means that there is no dominating drift
term. Further, it is well known that every $\psi$-irreducible Markov
process is either recurrent or transient (see \cite[Theorem
2.3]{tweedie-mproc}) and, clearly, every Harris recurrent Markov
process is recurrent but in general, these two properties are not
equivalent. They differ on the set of the irreducibility measure
zero (see \cite[Theorem 2.5]{tweedie-mproc}). However,
 for a Feller process satisfying
conditions (\textbf{C1})-(\textbf{C4}) these two
 properties are equivalent (see  Proposition
  \ref{p2.1}).

Throughout this paper, the symbol $\process{F}$ denotes a Feller
process satisfying conditions (\textbf{C1})-(\textbf{C4}). Such a
process is called a \emph{nice Feller process}.
 We say that $\process{F}$ is a
\emph{symmetric nice Feller process} if its corresponding symbol
satisfies $q(x,\xi)=\rm{Re}\,\it{q}(x,\xi)$, that is, if $b(x)=0$
and $\nu(x,dy)$ are symmetric measures for all $x\in\R^{d}.$ Also,
 a L\'evy process is denoted by  $\process{L}$.

The main result of this paper, the proof of which  is given in
Section \ref{s2}, is the following Chung-Fuchs type condition for
the recurrence  of nice Feller processes.
\begin{theorem}\label{tm1.2}Let  $\process{F}$ be a nice  Feller process  with  symbol $q(x,\xi)$. If $\rm{Re}\,\mathbb{E}^{0}[\it{e}^{\it{i}\langle\xi, F_t\rangle}]\geq\rm{0}$ for all $t\geq0$ and all $\xi\in\R^{d}$ and \be\label{eq:1.3}\int_{\{|\xi|<r\}}\frac{d\xi}{\sup_{x\in\R^{d}}|q(x,\xi)|}=\infty\quad\textrm{for some}\ r>0,\ee then $\process{F}$ is recurrent.
\end{theorem}
The  Chung-Fuchs type condition for transience of nice Feller
processes has been derived in \cite[Theorem 1.2]{rene-wang-feller}
and it reads as follows.
\begin{theorem}\label{tm1.3}Let  $\process{F}$ be a nice Feller process  with  symbol $q(x,\xi)$.
If $\process{F}$ satisfies the sector condition in $(\ref{eq:1.1})$
and
\be\label{eq:1.4}\int_{\{|\xi|<r\}}\frac{d\xi}{\inf_{x\in\R^{d}}\rm{Re}\,\it{q(x,\xi)}}<\infty\quad\textrm{for
some}\ r>0,\ee then $\process{F}$ is transient.
\end{theorem}
In the case  when $\process{F}$ is a L\'evy process  with  symbol
$q(\xi)$, the L\'evy-Khintchine formula yields
 $\mathbb{E}^{0}[\it{e}^{\it{i}\langle\xi, L_t\rangle}]=e^{-tq(\xi)}$ for all $t\geq0$ and all $\xi\in\R^{d}$ (see \cite[Theorems 7.10 and 8.1]{sato-book}). In particular,  if $\process{F}$ is a symmetric L\'evy process,  then
$\rm{Re}\,\mathbb{E}^{0}[\it{e}^{\it{i}\langle\xi,
L_t\rangle}]=e^{-tq(\xi)}\geq\rm{0}$. Thus,  we get the well-known
Chung-Fuchs conditions (see \cite[Theorem 37.5]{sato-book}). This
shows that the conditions of Theorems \ref{tm1.2} and \ref{tm1.3}
are sharp for L\'evy processes. Clearly,  for each frozen
$x\in\R^{d}$, $q(x,\xi)$ is the symbol of some L\'evy process
$\process{L^{x}}$. Thus, intuitively, Theorem \ref{tm1.2} says that
if all the L\'evy processes $\process{L^{x}}$, $x\in\R^{d}$, are
recurrent, then the Feller process $\process{F}$ is also recurrent.
Similarly, Theorem \ref{tm1.3} says that if all the L\'evy processes
$\process{L^{x}}$, $x\in\R^{d}$, are transient, then the Feller
process $\process{F}$ is also transient.

As is well known, the fact whether or not a L\'evy process is
recurrent or transient depends on the dimension of the state space.
Hence, it is natural to expect that a similar result holds in the
situation of nice Feller processes.  In Theorem \ref{tm2.8}, we
discuss this dependence, which again generalizes the L\'evy process
situation (see \cite[Theorems 37.8 and 37.14]{sato-book}). More
precisely, we prove that when $d=1,2$ and
$q(x,\xi)=\rm{Re}\,\it{q}(x,\xi)$ for all $x\in\R^{d}$, then
$$\sup_{x\in\R^{d}}\int_{\R^{d}}|y|^{2}\nu(x,dy)<\infty$$ implies
(\ref{eq:1.3}), and when $d\geq3$, then
$$\liminf_{|\xi|\longrightarrow0}\frac{\sup_{c>0}\inf_{x\in\R^{d}}\left(\langle\xi,c(x)\xi\rangle+\int_{\{|y|\leq
c\}}\langle\xi,y\rangle^{\rm{2}}\nu(\it{x},dy)\right)}{|\xi|^{2}}>\rm{0},$$
implies (\ref{eq:1.4}).  In particular, a symmetric
\emph{Feller-Dynkin diffusion}, that is, a symmetric nice Feller
process determined by a symbol of the form
$q(x,\xi)=\frac{1}{2}\langle\xi,c(x)\xi\rangle,$ is recurrent if,
and only if, $d=1,2$ (see Theorem \ref{tm2.9}).

Recall that a \emph{rotationally invariant stable L\'evy process}
$\process{L}$ is a L\'evy process with symbol given by
$q(\xi)=\gamma|\xi|^{\alpha}$, where  $\alpha\in(0,2]$ and
$\gamma\in(0,\infty)$. The parameters $\alpha$ and $\gamma$ are
called the stability parameter and the scaling parameter,
respectively (see \cite[Chapter 3]{sato-book} for details). Note
that when $\alpha=2$, then $\process{L}$ becomes a Brownian motion.
It is well known that the recurrence and transience property of
$\process{L}$ depends on the index of stability $\alpha$. More
precisely, if $d\geq3$, then $\process{L}$ is always transient, if
$d=2$, then $\process{L}$ is recurrent if, and only if, $\alpha=2$
and if $d=1$, then $\process{L}$ is recurrent if, and only if,
$\alpha\geq1$ (see \cite[Theorems 37.8, 37.16 and
37.18]{sato-book}). The notion of stable L\'evy processes has been
generalized  in \cite{bass-stablelike}. More precisely, under some
technical assumptions on the functions
$\alpha:\R^{d}\longrightarrow(0,2)$ and
$\gamma:\R^{d}\longrightarrow(0,\infty)$ (see Section \ref{s2} for
details),  in \cite{bass-stablelike} and  \cite[Theorem
3.3]{rene-wang-feller} it has been shown that there exists a unique
nice Feller process, called a \emph{stable-like process}, determined
by a symbol of the form $q(x,\xi)=\gamma(x)|\xi|^{\alpha(x)}$. In
Theorem \ref{tm2.10} and Corollary \ref{c3.3}, we discuss the
recurrence and transience property of stable-like processes. Next,
the concept of the indices of stability can be generalized to
general L\'evy processes through the Pruitt indices (see
\cite{pruitt}). The Pruitt indices for nice Feller processes have
been introduced in \cite{rene-holder}. In Theorems \ref{tm2.12} and
\ref{tm2.13}, we also discuss  the recurrence and transience
property of nice Feller processes, as well as of L\'evy processes,
in terms of the Pruitt indices.

A natural problem which arises is to determine  those perturbations
of the symbol (or the L\'evy quadruple) which will not affect the
recurrence or transience property of the underlying nice Feller
process. In the one-dimensional symmetric case, in Theorem
\ref{tm3.1}, we prove that if $\process{F^{1}}$ and
$\process{F^{2}}$ are two nice Feller process with  L\'evy measures
$\nu_1(x,dy)$ and $\nu_2(x,dy)$, respectively, such that
$$\sup_{x\in\R}\int_{\R}y^{2}|\nu_1(x,dy)-\nu_2(x,dy)|<\infty,$$
then $\process{F^{1}}$  and $\process{F^{2}}$ are recurrent or
transient at the same time. Here, $|\mu(dy)|$ denotes the total
variation measure of the signed measure $\mu(dy)$. In particular,
we conclude that the recurrence and transience property of
one-dimensional symmetric nice Feller processes depends only on big
jumps. Further, in general it is not always easy to compute the
Chung-Fuchs type conditions in (\ref{eq:1.3}) and (\ref{eq:1.4}).
According to this, in the one-dimensional symmetric case, in
Theorems \ref{tm3.7} and \ref{tm3.9}, we give necessary and
sufficient condition for the recurrence and transience in terms of
the L\'evy measure. Finally, in Theorems \ref{tm3.12} and
\ref{tm3.13}, we give some comparison conditions for the recurrence
and transience  in terms of the L\'evy measures.

In the L\'evy process case, the main ingredient in the proof of the
Chung-Fuchs conditions is the fact that
$$\mathbb{E}^{x}[e^{i\langle\xi,L_t-x\rangle}]=e^{-tq(\xi)}$$ for all $t\geq0$ and all $x,\xi\in\R^{d}$, where
$q(\xi)$ is the symbol of the L\'evy process $\process{L}$ (see
\cite[Theorems 7.10 and 8.1]{sato-book}). This relation is no longer
true for a general nice Feller process $\process{F}$. Since
$\process{F}$ does not have stationary and independent increments,
in particular, it is
 not spatially homogeneous, the characteristic function of $F_t$, $t\geq0$, will
      now depend on the starting point $x\in\R^{d}$ and $q(x,\xi)$ is no longer the characteristic exponent of $\process{F}$. Anyway, it is
      natural to expect that  $$\mathbb{E}^{x}[e^{i\langle\xi,F_t-x\rangle}]\approx
      e^{-tq(x,\xi)}$$  for all $t\geq0$ and all $x,\xi\in\R^{d}$.
      According to this, as the main step in the proof of Theorem \ref{tm1.2}
      we derive a lower bound for $\mathbb{E}^{\it{x}}[\it{e}^{\it{i}\langle\xi,
      F_t-x\rangle}]$. More precisely, in Lemma \ref{l2.2}, we prove  that for any $\varepsilon>0$ and $\xi\in\R^{d}$ there exists
$t_0:=t_0(\varepsilon,\xi)>0$, such that for all $x\in\R^{d}$ and
all $t\in[0,t_0]$ we have
$$\rm{Re}\,\mathbb{E}^{\it{x}}[\it{e}^{\it{i}\langle\xi,
      F_t-x\rangle}]\geq\exp\left[-(\rm{2}+\varepsilon)\it{t}\sup_{\it{z}\in\R^{d}}|\it{q}(\it{z},\xi)|\right].$$
The upper bound for $\mathbb{E}^{x}[\it{e}^{\it{i}\langle\xi,
F_t-x\rangle}]$
 has been
given in \cite[Theorem 2.7]{rene-wang-feller} and it reads as
follows
$$\left|\mathbb{E}^{x}[\it{e}^{\it{i}\langle\xi, F_t-x\rangle}]\right|\leq\exp\left[-\frac{t}{\rm{16}}\inf_{z\in\R^{d}}\rm{Re}\,\it{q}(z,\rm{2}\it{\xi})\right]$$ for all $t\geq0$ and all $x,\xi\in\R^{d}$.
The proofs of the remaining  results presented in this paper are
mostly based on the Chung-Fuchs type conditions in (\ref{eq:1.3})
and (\ref{eq:1.4}) and the analysis of the symbols.

The sequel of this paper is organized as follows. In Section
\ref{s2}, we prove Theorem \ref{tm1.2} and  discuss  the recurrence
and transience  of nice Feller processes with respect to the
dimension of the state space and Pruitt indices and  the recurrence
and transience of Feller-Dynkin diffusions and  stable-like
processes. Finally, in Section \ref{s3}, we discuss the recurrence
and transience property of one-dimensional symmetric nice Feller
processes. More precisely, we study perturbations of  nice Feller
processes and we derive sufficient conditions for the recurrence and
transience in terms of the  L\'evy measure and  give some comparison
conditions for the recurrence and transience property in terms of
the L\'evy measures.

\section{Recurrence and transience of general nice Feller
processes}\label{s2}

\ \ \ \ We start this section with some preliminary and auxiliary
results regarding the recurrence and transience property of nice
Feller processes which  we need in the sequel. First, recall that a
semigroup $\process{P}$ on $(B_b(\R^{d}),||\cdot||_\infty)$ is
called a \emph{$C_b$-Feller semigroup} if $P_t(C_b(\R^{d}))\subseteq
C_b(\R^{d})$ for all $t\geq0$ and it is called a \emph{strong Feller
semigroup} if $P_t(B_b(\R^{d}))\subseteq C_b(\R^{d})$ for all
$t\geq0$. Here, $C_b(\R^{d})$  denotes the space of continuous and
bounded functions. For sufficient conditions for a Feller semigroup
to be a $C_b$-Feller semigroup or a strong Feller semigroup see
\cite{rene-conserv} and \cite{rene-wang-strong}.
\begin{proposition}\label{p2.1} Let $\process{F}$ be a nice Feller
process. Then the following properties are equivalent:
\begin{enumerate}
\item [(i)] $\process{F}$ is recurrent
\item[(ii)] $\process{F}$ is  Harris
recurrent
\item[(iii)]  there exists
$x\in\R^{d}$ such that
$$\mathbb{P}^{x}\left(\liminf_{t\longrightarrow\infty}|F_t-x|=0\right)=1$$
\item [(iv)]  there exists $x\in\R^{d}$ such that  $$\int_{0}^{\infty}\mathbb{P}^{x}(F_t\in O_x)dt=\infty$$ for
  all open neighborhoods $O_x\subseteq\R^{d}$ around $x$
\item [(v)]  there exists  a compact set $C\subseteq\R^{d}$ such that $$\mathbb{P}^{x}(\tau_C<\infty)=1$$ for all
                      $x\in\R^{d}$
                        \item [(vi)]  for each initial position $x\in\R^{d}$ and each   covering
$\{O_n\}_{n\in\N}$ of $\R^{d}$ by open bounded sets we have
$$\mathbb{P}^{x}\left(\bigcap_{n=1}^{\infty}\bigcup_{m=0}^{\infty}\left\{\int_m^{\infty}1_{\{F_t\in
O_n\}}dt=0\right\}\right)=0.$$ In other words, $\process{F}$ is
recurrent if, and only if, for each initial position $x\in\R^{d}$
the event $\{F_t\in C^{c}\ \textrm{for any compact set}\
C\subseteq\R^{d}\ \textrm{and all}\ t\geq0\ \textrm{sufficiently
large}\}$ has probability $0$.
\end{enumerate}
In addition, if we assume that $\process{F}$ is a strong Feller
process, then  all the statements above are also equivalent to:
\begin{enumerate}
  \item [(vii)] there exists a compact set $C\subseteq\R^{d}$ such that $$\int_{0}^{\infty}\mathbb{P}^{x}(F_t\in C)dt=\infty$$ for
  all $x\in\R^{d}$
  \item [(viii)]  there
  exist $x\in\R^{d}$ and an open bounded set $O\subseteq\R^{d}$ such
  that $$\mathbb{P}^{x}\left(\int_0^{\infty}1_{\{F_t\in
  O\}}dt=\infty\right)>0.$$
\end{enumerate}
\end{proposition}
\begin{proof}
First, let us remark that  every Feller semigroup $\{P_t\}_{t\geq0}$
has a unique extension onto the space $B_b(\R^{d})$ (see
\cite[Section 3]{rene-conserv}). For notational simplicity, we
denote this extension again by $\{P_t\}_{t\geq0}$. Now, according to
\cite[Corollary 3.4 and Theorem 4.3]{rene-conserv} and \cite[Lemma
2.3]{rene-wang-feller}, $\{P_t\}_{t\geq0}$ is a $C_b$-Feller
semigroup.
\begin{description}
\item[$(i)\Leftrightarrow(ii)$] This is an immediate consequence of \cite[Theorems 4.1, 4.2 and
  7.1]{tweedie-mproc}.
\item[$(i)\Leftrightarrow(iii)$] This is an immediate consequence of  \cite[Theorem
  7.1]{tweedie-mproc}  and \cite[Theorem
4.3]{bjoern-overshoot}.
\item[$(i)\Leftrightarrow(iv)$] This is an immediate consequence of \cite[Theorems
4.1 and 7.1]{tweedie-mproc}.
\item[$(i)\Leftrightarrow(v)$] This is an immediate consequence of (iii) and \cite[Theorem
  3.3]{meyn-tweedie-mproc}.
    \item[$(i)\Leftrightarrow(vi)$] This is an immediate consequence of (ii) and \cite[Theorem
  3.3]{tweedie-mproc}.
  \item [$(i)\Leftrightarrow(vii)$] By \cite[Proposition 2.3]{sandric-periodic}, it suffices to
  prove that $$\inf_{x\in C}\int_0^{\infty}\mathbb{P}^{x}(F_t\in
  B)e^{-t}dt>0$$ holds for every compact set $C\subseteq\R^{d}$ and
  every $B\in\mathcal{B}(\R^{d})$ satisfying $\psi(B)>0$. Let us assume
  that
  this is not the case. Then, there exist a compact set $C\subseteq\R^{d}$, a Borel set
  $B\subseteq\R^{d}$ satisfying
  $\psi(B)>0$ and a sequence $\{x_n\}_{n\in\N}\subseteq C$ with $\lim_{n\longrightarrow\infty}x_n=x_0\in C$, such
  that $$\lim_{n\longrightarrow\infty}\int_0^{\infty}\mathbb{P}^{x_n}(F_t\in
  B)e^{-t}dt=0.$$ Now, by the
  dominated convergence theorem and the strong Feller property, it follows $$\int_0^{\infty}\mathbb{P}^{x_0}(F_t\in
  B)e^{-t}dt=0.$$ But this is in contradiction with the
  $\psi$-irreducibility property of $\process{F}$.
  \item[$(i)\Leftrightarrow(viii)$] This is an immediate consequence of (vii) and
  \cite[Poposition 2.4]{sandric-periodic}.
\end{description}
\end{proof}
The proof of Theorem \ref{tm1.2} is based on the following lemma.
\begin{lemma}\label{l2.2}Let $\process{F}$ be a nice   Feller process  with  symbol $q(x,\xi)$ and let
$\Phi_t(x,\xi):=\mathbb{E}^{x}\left[e^{i\langle\xi,F_t-x\rangle}\right]$
for $t\geq0$ and $x,\xi\in\R^{d}.$ Then, for any $\varepsilon>0$ and
$\xi\in\R^{d}$ there exists $t_0:=t_0(\varepsilon,\xi)>0$, such that
for all $x\in\R^{d}$ and all $t\in[0,t_0]$ we have
$$\rm{Re}\,\Phi_{\it{t}}(\it{x},\xi)\geq\exp\left[-(\rm{2}+\varepsilon)\it{t}\sup_{\it{z}\in\R^{d}}|\it{q}(\it{z},\xi)|\right].$$
\end{lemma}
\begin{proof} First,
by \cite[Proposition 4.2]{rene-wang-feller} and \cite[proof of Lemma
6.3]{rene-holder}, for all $t\geq0$ and all $x,\xi\in\R^{d}$ we have
$$\Phi_t(x,\xi)=1-\int_0^{t}P_s\left(q(\cdot,\xi)e^{i\langle\xi,
\cdot-x\rangle}\right)(x)ds.$$ Recall that $\process{P}$ denotes the
semigroup of $\process{F}$. Thus,
\begin{align*}\rm{Re}\,\Phi_{\it{t}}(\it{x},\xi)&=1-\int_0^{t}\int_{\R^{d}}\left(\cos\langle\xi,
y-x\rangle\rm{Re}\,\it{q}(y,\xi)-\sin\langle\xi,
y-x\rangle\rm{Im}\,\it{q}(y,\xi)\right)\mathbb{P}^{x}(F_s\in
dy)ds\\&
\geq1-\int_0^{t}\int_{\R^{d}}\left(\rm{Re}\,\it{q}(y,\xi)+|\rm{Im}\,\it{q}(y,\xi)|\right)\mathbb{P}^{x}(F_s\in
dy)ds\\&
\geq\rm{1}-2\it{t}\sup_{\it{z}\in\R^{d}}|q(\it{z},\xi)|.\end{align*}
Finally, for given $\varepsilon>0$ and $\xi\in\R^{d}$, it is easy to
check that for all
$t\in\left[0,\frac{\ln\left(\frac{2+\varepsilon}{2}\right)}{(2+\varepsilon)\sup_{z\in\R^{d}}|q(z,\xi)|}\right]$
we have
$$\rm{Re}\,\Phi_{\it{t}}(\it{x},\xi)\geq\rm{1}-2\it{t}\sup_{z\in\R^{d}}|q(z,\xi)|\geq\exp\left[-(\rm{2}+\varepsilon)\it{t}\sup_{z\in\R^{d}}|q(z,\xi)|\right].$$
\end{proof}      Now, we prove  Theorem
\ref{tm1.2}.
\begin{proof}[Proof of Theorem \ref{tm1.2}]
First, note that, according to Proposition \ref{p2.1} (iv), it
suffices to prove that
$$\mathbb{E}^{0}\left[\int_0^{\infty}1_{\{F_t\in O_0\}}dt\right]=\infty$$
for every open neighborhood $O_0\subseteq\R^{d}$ around the origin.
Let $a>0$ be arbitrary. By the monotone convergence theorem, we have
\begin{align*}\mathbb{E}^{0}\left[\int_0^{\infty}1_{\left\{F_t\in
\left(-a,a\right)^{d}\right\}}dt\right]&=\lim_{\alpha\longrightarrow0}\mathbb{E}^{0}\left[\int_0^{\infty}e^{-\alpha
t}1_{\left\{F_t\in
\left(-a,a\right)^{d}\right\}}dt\right]\\&=\lim_{\alpha\longrightarrow0}\int_0^{\infty}\int_{\R^{d}}e^{-\alpha
t}1_{ \left(-a,a\right)^{d}}(y)\mathbb{P}^{0}(F_t\in
dy)dt,\end{align*} where
$\left(-a,a\right)^{d}:=\left(-a,a\right)\times\ldots\times\left(-a,a\right)$.
Next, let
$$f(x):=\left(1-\frac{|x|}{a}\right)1_{\left(-a,a\right)}(x),$$ for $x\in\R$, and
$$g(x):=f(x_1)\cdots f(x_d),$$ for $x=(x_1,\ldots,x_d)\in\R^{d}$.
Clearly,  we have $$ 1_{ \left(-a,a\right)^{d}}(x)\geq g(x)$$ for
all $x\in\R^{d}$. According to this, we have
\be\label{eq:2.1}\mathbb{E}^{0}\left[\int_0^{\infty}1_{\left\{F_t\in
\left(-a,a\right)^{d}\right\}}dt\right]\geq\liminf_{\alpha\longrightarrow0}\int_0^{\infty}\int_{\R^{d}}e^{-\alpha
t}g(y)\mathbb{P}^{0}(F_t\in dy)dt.\ee Further,
$$f(x)=\frac{1}{\sqrt{a}}1_{\left(-\frac{a}{2},\frac{a}{2}\right)}\ast\frac{1}{\sqrt{a}}1_{\left(-\frac{a}{2},\frac{a}{2}\right)}(x),$$
where $\ast$ denotes the standard convolution operator. Hence, since
$$\mathcal{F}\left(\frac{1}{\sqrt{a}}1_{\left(-\frac{a}{2},\frac{a}{2}\right)}\right)(\xi)=\frac{\sin\left(\frac{a\xi}{2}\right)}{\sqrt{a}\pi
\xi},$$ we have
$$\mathcal{F}(g)(\xi)=\frac{\sin^{2}\left(\frac{a\xi_1}{2}\right)}{a\pi^{2}
\xi_1^{2}}\cdots\frac{\sin^{2}\left(\frac{a\xi_d}{2}\right)}{a\pi^{2}
\xi_d^{2}}.$$ This and (\ref{eq:2.1}) yields
\begin{align*}\mathbb{E}^{0}\left[\int_0^{\infty}1_{\left\{F_t\in
\left(-a,a\right)^{d}\right\}}dt\right]&\geq\liminf_{\alpha
\longrightarrow0}\int_0^{\infty}\int_{\R^{d}}\int_{\R^{d}}e^{-\alpha
t}e^{i\langle\xi,y\rangle}\mathcal{F}(g)(\xi)d\xi
\mathbb{P}^{0}(F_t\in
dy)dt\\&=\liminf_{\alpha\longrightarrow0}\int_0^{\infty}\int_{\R^{d}}e^{-\alpha
t}\Phi_t(0,\xi)\mathcal{F}(g)(\xi)d\xi
dt\\&=\liminf_{\alpha\longrightarrow0}\int_0^{\infty}\int_{\R^{d}}e^{-\alpha
t}\rm{Re}\,\Phi_{\it{t}}(\rm{0},\xi)\mathcal{F}(\it{g})(\xi)d\xi
d\it{t}\\&=\liminf_{\alpha\longrightarrow0}\int_{\R^{d}}\Bigg(\int_0^{t_0(2,\xi)}e^{-\alpha
t}\rm{Re}\,\Phi_{\it{t}}(\rm{0},\xi)d\it{t}\\& \ \ \ \ \ \ \ \ \ \ \
\ \ \ \ \ \  \
+\int_{t_{\rm{0}}(\rm{2},\xi)}^{\infty}e^{-\alpha\it{t}}\rm{Re}\,\Phi_{\it{t}}(\rm{0},\xi)d\it{t}\Bigg)\mathcal{F}(g)(\xi)d\xi
\\&\geq
\liminf_{\alpha\longrightarrow0}\int_{\R^{d}}\frac{1-\exp\left[-\ln2\frac{\alpha+4\sup_{x\in\R^{d}}|q(x,\xi)|}{4\sup_{x\in\R^{d}}|q(x,\xi)|}\right]}{\alpha+4\sup_{x\in\R^{d}}|q(x,\xi)|}\mathcal{F}(g)(\xi)d\xi,
\end{align*}
where in the fourth step
$t_0(2,\xi)=\frac{\ln2}{4\sup_{x\in\R^{d}}|q(x,\xi)|}$ is given in
Lemma \ref{l2.2} and in the final step we applied Lemma \ref{l2.2}
and the assumption that $\rm{Re}\,\Phi_{\it{t}}(\rm{0},\xi)\geq0$
for all $t\geq0$ and all $\xi\in\R^{d}.$ Now, by  Fatou's lemma, we
have
$$\mathbb{E}^{0}\left[\int_0^{\infty}1_{\left\{F_t\in
\left(-a,a\right)^{d}\right\}}dt\right]\geq\frac{1}{8}\int_{\R^{d}}\frac{\mathcal{F}(g)(\xi)}{\sup_{x\in\R^{d}}|q(x,\xi)|}d\xi.$$
Finally, let $r>0$ be such that
$$\int_{\{|\xi|<r\}}\frac{d\xi}{\sup_{x\in\R^{d}}|q(x,\xi)|}=\infty.$$ Then, since
$$\lim_{a\longrightarrow0}\frac{(2\pi)^{2d}}{a^{d}}\mathcal{F}(g)(\xi)=1,$$ for any $c\in(0,1)$, all $|\xi|<r$ and all $a>0$ small enough
we have $$\mathbb{E}^{0}\left[\int_0^{\infty}1_{\left\{F_t\in
\left(-a,a\right)^{d}\right\}}dt\right]\geq\frac{ca^{d}}{8(2\pi)^{2d}}\int_{\{|\xi|<r\}}\frac{d\xi}{\sup_{x\in\R^{d}}|q(x,\xi)|}=\infty,$$
which completes the proof.
\end{proof}
As a consequence, we also get the following Chung-Fuchs type
conditions.
\begin{corollary}Let  $\process{F}$ be a nice Feller process  with  symbol $q(x,\xi)$ satisfying $|\rm{Im}\,\it{q}(x,\xi)|\leq c\,\rm{Re}\,\it{q}(x,\xi)$ for some
$c\geq0$ and all $x,\xi\in\R^{d}$. Then, \begin{itemize}
                                                                                           \item[(i)]
 the condition in (\ref{eq:1.3}) is equivalent with
$$\int_{\{|\xi|<r\}}\frac{d\xi}{\sup_{x\in\R^{d}}\rm{Re}\,\it{q}(\it{x},\xi)}=\infty\quad\textrm{for
some}\ r>0.$$
                                                                                          \item [(ii)]    the condition in (\ref{eq:1.4}) is equivalent with
$$\int_{\{|\xi|<r\}}\frac{d\xi}{\inf_{x\in\R^{d}}|p(x,\xi)|}<\infty\quad\textrm{for
some}\ r>0.$$
\end{itemize}
\end{corollary}
\begin{proof}
The desired results easily follow from the following inequalities
$$\frac{\bar{c}}{\rm{Re}\,\it{q}(x,\xi)}\leq\frac{1}{\sqrt{(\rm{Re}\,\it{q}(x,\xi))^{\rm{2}}+(\rm{Im}\,\it{q}(x,\xi))^{\rm{2}}}}=\frac{1}{|q(x,\xi)|}\leq\frac{1}{\rm{Re}\,\it{q}(x,\xi)},$$
where $\bar{c}=\frac{1}{\sqrt{1+c^{2}}}.$
\end{proof}

 In the following proposition we discuss the
dependence of the conditions in (\ref{eq:1.3}) and (\ref{eq:1.4}) on
$r>0$. First, note that if the condition in (\ref{eq:1.4}) holds for
some $r_0>0$, then it  also holds for all $0<r\leq r_0$. In
addition,
 if we assume that \be\label{eq:1}\inf_{r_0\leq|\xi|\leq
r}\inf_{x\in\R^{d}}\rm{Re}\,\it{q}(x,\xi)>\rm{0}\ee holds for all
$0<r_0<r$, then the condition in (\ref{eq:1.4}) does not depend on
$r>0$. In particular, (\ref{eq:1}) is satisfied if
$$\inf_{|\xi|=1}\inf_{x\in\R^{d}}\left(\langle\xi,c(x)\xi\rangle+\int_{\{|y|\leq\frac{1}{r}\}}\langle\xi,y\rangle^{2}\nu(x,dy)\right)>0$$
holds for all $r>r_0$. Indeed, let $r>r_0$ be arbitrary. Then, we
have
\begin{align*}\inf_{r_0\leq|\xi|\leq
r}\inf_{x\in\R^{d}}\rm{Re}\,\it{q}(x,\xi)&=\inf_{r_0\leq|\xi|\leq
r}\inf_{x\in\R^{d}}\left(\frac{1}{2}\langle\xi,c(x)\xi\rangle+\int_{\R^{d}}(1-\cos\langle\xi,y\rangle)\nu(x,dy)\right)\\
&\geq\inf_{r_0\leq|\xi|\leq
r}\inf_{x\in\R^{d}}\left(\frac{1}{2}\langle\xi,c(x)\xi\rangle+\int_{\{|\xi||y|\leq1\}}(1-\cos\langle\xi,y\rangle)\nu(x,dy)\right)\\
&\geq\frac{1}{\pi}\inf_{r_0\leq|\xi|\leq r}\inf_{x\in\R^{d}}\left(\langle\xi,c(x)\xi\rangle+\int_{\{|\xi||y|\leq1\}}\langle\xi,y\rangle^{2}\nu(x,dy)\right)\\
&\geq\frac{r^{2}_0}{\pi}\inf_{r_0\leq|\xi|\leq
r}\inf_{x\in\R^{d}}\left(\left\langle\frac{\xi}{|\xi|},\frac{c(x)}{|\xi|}\xi\right\rangle+\int_{\left\{|y|\leq\frac{1}{r}\right\}}\left\langle\frac{\xi}{|\xi|},y\right\rangle^{2}\nu(x,dy)\right)\\
&\geq\frac{r_0^{2}}{\pi}\inf_{|\xi|=1}\inf_{x\in\R^{d}}\left(\langle\xi,c(x)\xi\rangle+\int_{\{|y|\leq\frac{1}{r}\}}\langle\xi,y\rangle^{2}\nu(x,dy)\right),\end{align*}
where in the third step we employed the fact that $1-\cos
y\geq\frac{1}{\pi}y^{2}$ for all $|y|\leq\frac{\pi}{2}.$
\begin{proposition}\label{p2.4}Let $\process{F}$ be a nice Feller process with
 symbol $q(x,\xi).$
 If the functions
$\xi\longmapsto\sup_{x\in\R^{d}}|q(x,\xi)|$ and
$\xi\longmapsto\inf_{x\in\R^{d}}\rm{Re}\,\it{q}(x,\xi)$ are radial
and the function
$\xi\longmapsto\inf_{x\in\R^{d}}\sqrt{\rm{Re}\,\it{q}(x,\xi)}$ is
subadditive, then the conditions in (\ref{eq:1.3}) and
(\ref{eq:1.4}) do not depend on $r>0$.
\end{proposition}
\begin{proof}
 First, we prove that the functions
$\xi\longmapsto\sup_{x\in\R^{d}}|q(x,\xi)|$ and
$\xi\longmapsto\inf_{x\in\R^{d}}\rm{Re}\,\it{q}(x,\xi)$ are
continuous. Let $\xi,\xi_0\in\R^{d}$ be arbitrary. By \cite[Lemma
3.6.21]{jacobI}, we have
\begin{align*}\left|\sqrt{\sup_{x\in\R^{d}}|q(x,\xi)|}-\sqrt{\sup_{x\in\R^{d}}|q(x,\xi_0)|}\right|&\leq\sup_{x\in\R^{d}}\left|\sqrt{|q(x,\xi)|}-\sqrt{|q(x,\xi_0)|}\right|\\&\leq
\sup_{x\in\R^{d}}\sqrt{|q(x,\xi-\xi_0)|}\end{align*} and
\begin{align*}\left|\sqrt{\inf_{x\in\R^{d}}\rm{Re}\,\it{q}(x,\xi)}-\sqrt{\inf_{x\in\R^{d}}\rm{Re}\,\it{q}(x,\xi_{\rm{0}})}\right|&\leq\sup_{x\in\R^{d}}\left|\sqrt{\rm{Re}\,\it{q}(x,\xi)}-\sqrt{\rm{Re}\,\it{q}(x,\xi_{\rm{0}})}\right|\\&\leq
\sup_{x\in\R^{d}}\sqrt{\rm{Re}\,\it{q}(x,\xi-\xi_{\rm{0}})}\\&\leq\sup_{x\in\R^{d}}\sqrt{|\it{q}(x,\xi-\xi_{\rm{0}})|}.\end{align*}
Now, by letting $\xi\longrightarrow\xi_0$, the claim follows from
\cite[Theorem 4.4]{rene-conserv}.

Let us first consider the recurrence case. Let $r_0>0$ be such that
$$\int_{\{|\xi|<r_0\}}\frac{d\xi}{\sup_{x\in\R^{d}}|q(x,\xi)|}<\infty$$
and
$$\int_{\{|\xi|<r\}}\frac{d\xi}{\sup_{x\in\R^{d}}|q(x,\xi)|}=\infty$$
for all $r>r_0$. In particular, we have
$$\int_{\{r_0\leq|\xi|\leq
r\}}\frac{d\xi}{\sup_{x\in\R^{d}}|q(x,\xi)|}=\infty$$ for all
$r>r_0$. Let $r>r_0$ be arbitrary. Then, by compactness, there
exists a sequence $\{\xi_n\}_{n\in\N}\subseteq\{\xi:r_0\leq|\xi|\leq
r\}$, such that $\xi_n\longrightarrow\xi_r\in\{\xi:r_0\leq|\xi|\leq
r\}$ and
$\lim_{n\longrightarrow\infty}\sup_{x\in\R^{d}}|q(x,\xi_n)|=0$. In
particular, by the continuity, $\sup_{x\in\R^{d}}|q(x,\xi_r)|=0$.
Since this is true for every $r>r_0$, we have
$\lim_{r\longrightarrow r_0}|\xi_r|= r_0$. Thus, by  the continuity
and  radial property, for arbitrary $\xi_0\in\R^{d}$, $|\xi_0|=r_0$,
we have $\sup_{x\in\R^{d}}|q(x,\xi_0)|=0$. Now, since
$\xi\longmapsto\sqrt{|q(x,\xi)|}$ is subadditive for all
$x\in\R^{d}$ (see \cite[Lemma 3.6.21]{jacobI}), by the radial
property, for arbitrary $\xi\in\R^{d}$, we have
\begin{align*}\sup_{x\in\R^{d}}\sqrt{\left|q\left(x,\xi+\xi_0\right)\right|}&\leq\sup_{x\in\R^{d}}\left(\sqrt{\left|q\left(x,\xi\right)\right|}+\sqrt{\left|q\left(x,\xi_0\right)\right|}\right)\\&\leq\sup_{x\in\R^{d}}\sqrt{\left|q\left(x,\xi\right)\right|}+\sup_{x\in\R^{d}}\sqrt{\left|q\left(x,\xi_0\right)\right|}\\&=\sup_{x\in\R^{d}}\sqrt{\left|q\left(x,\xi\right)\right|}\end{align*}
and
\begin{align*}\sup_{x\in\R^{d}}\sqrt{\left|q\left(x,\xi\right)\right|}&=\sup_{x\in\R^{d}}\sqrt{\left|q\left(x,\xi+\xi_0-\xi_0\right)\right|}\\&\leq\sup_{x\in\R^{d}}\left(\sqrt{\left|q\left(x,\xi+\xi_0\right)\right|}+\sqrt{\left|q\left(x,-\xi_0\right)\right|}\right)\\&\leq\sup_{x\in\R^{d}}\sqrt{\left|q\left(x,\xi+\xi_0\right)\right|}+\sup_{x\in\R^{d}}\sqrt{\left|q\left(x,-\xi_0\right)\right|}\\&=\sup_{x\in\R^{d}}\sqrt{\left|q\left(x,\xi+\xi_0\right)\right|},\end{align*}
that is,  the function $\xi\longmapsto\sup_{x\in\R^{d}}|q(x,\xi)|$
is periodic with period $\xi_0$. Thus, we conclude that if
(\ref{eq:1.3}) holds for some $r>0$, then it holds for all $r>0$.

In the transience case, by completely the same arguments as above,
we have that
\\$\inf_{x\in\R^{d}}\rm{Re}\,\it{q}(x,\xi_{\rm{0}})=\rm{0}$ for all
$\xi_0\in\R^{d}$, $|\xi_0|=r_0$, where $r_0> 0$ is such that
$$\int_{\{|\xi|<r_0\}}\frac{d\xi}{\inf_{x\in\R^{d}}\rm{Re}\,\it{q}(x,\xi)}<\infty$$
and
$$\int_{\{|\xi|<r\}}\frac{d\xi}{\inf_{x\in\R^{d}}\rm{Re}\,\it{q}(x,\xi)}=\infty$$
for all $r>r_0$. Further, since we assumed that the function
$\xi\longmapsto\inf_{x\in\R^{d}}\sqrt{\rm{Re}\,\it{q}(x,\xi)}$ is
subadditive, analogously as above, we conclude that the function
$\xi\longmapsto\inf_{x\in\R^{d}}\rm{Re}\,\it{q}(x,\xi)$ is periodic
with period $\xi_0$. Thus, if (\ref{eq:1.4}) holds for some $r>0$,
then it holds for all $r>0$. This completes the proof of Proposition
\ref{p2.4}.
\end{proof}
Let us remark that, in the one-dimensional case, a sufficient
condition for the  subadditivity  of the function
$\xi\longmapsto\inf_{x\in\R}\sqrt{\rm{Re}\,\it{q}\left(x,\xi\right)}$
is the concavity of  the function
$|\xi|\longmapsto\inf_{x\in\R}\sqrt{\rm{Re}\,\it{q}\left(x,\xi\right)}.$
Indeed, let $\xi,\eta>0$  be arbitrary. Then, we have
\begin{align*}\inf_{x\in\R}\sqrt{\rm{Re}\,\it{q}(x,\xi)}
&=\inf_{x\in\R}\sqrt{\rm{Re}\,\it{q}\left(x,\frac{\eta}{\xi+\eta}\rm{0}+\frac{\xi}{\xi+\eta}(\xi+\eta)\right)}\\
&\geq\frac{\xi}{\xi+\eta}\inf_{x\in\R}\sqrt{\rm{Re}\,\it{q}\left(x,\xi+\eta\right)}\end{align*}
and similarly
$$\inf_{x\in\R}\sqrt{\rm{Re}\,\it{q}(x,\eta)}\geq\frac{\eta}{\xi+\eta}\inf_{x\in\R}\sqrt{\rm{Re}\,\it{q}\left(x,\xi+\eta\right)}.$$
Thus,
\begin{align*}\inf_{x\in\R}\sqrt{\rm{Re}\,\it{q}\left(x,\xi+\eta\right)}\leq\inf_{x\in\R}\sqrt{\rm{Re}\,\it{q}(x,\xi)}+\inf_{x\in\R}\sqrt{\it{q}(x,\eta)},
\end{align*} that is, the function
$[0,\infty)\ni\xi\longmapsto\inf_{x\in\R^{d}}\sqrt{\rm{Re}\,\it{q}(x,\xi)}$
is subadditive. Finally, since every nonnegative and concave
function is necessarily nondecreasing, the claim follows.

As we commented in the first section, in the case when a symbol
$q(x,\xi)$ does not depend on the variable $x\in\R^{d}$,
$\process{F}$ becomes a L\'evy process and, by the L\'evy-Khintchine
formula, we have
$$q(\xi):=q(x,\xi)=-\frac{\log\mathbb{E}^{x}\left[e^{i\langle\xi-x,F_t\rangle}\right]}{t}=-\frac{\log\mathbb{E}^{0}\left[e^{i\langle\xi,F_t\rangle}\right]}{t}$$ and $\Phi_{\it{t}}(x,\xi)=e^{-tq(\xi)}$ for all $t\geq0$ and all $x,\xi\in\R^{d}$ (see \cite[Theorems 7.10 and 8.1]{sato-book}). Further, note that every
L\'evy process satisfies conditions (\textbf{C1})-(\textbf{C3}) (see
\cite[Theorem 31.5]{sato-book}) and if $\process{F}$ is a symmetric
L\'evy process,  then
$\rm{Re}\,\Phi_{\it{t}}(\it{x},\xi)=\rm{\Phi}_{\it{t}}(\it{x},\xi)=e^{-tq(\xi)}\geq\rm{0}$.
Thus, under condition (\textbf{C4}), we get the following well-known
Chung-Fuchs conditions (see \cite[Theorem 37.5]{sato-book}).
\begin{corollary}
Let $\process{F}$ be a  L\'evy process with  symbol $q(\xi)$ which
satisfies condition (\textbf{C4}). If $\process{F}$ is symmetric and
if
$$\int_{\{|\xi|<r\}}\frac{d\xi}{q(\xi)}=\infty\quad\textrm{for
some}\ r>0,$$ then $\process{F}$ is recurrent. If  $q(\xi)$
satisfies the sector condition in (\ref{eq:1.1}) and
$$\int_{\{|\xi|<r\}}\frac{d\xi}{\rm{Re}\,\it{q(\xi)}}<\infty\quad\textrm{for
some}\ r>0,$$ then $\process{F}$ is transient.
\end{corollary}
Let us remark that in general, because of stationary and independent
increments, the notion of irreducibility, and therefore condition
(\textbf{C4}), is not needed to derive the recurrence and transience
dichotomy of L\'evy processes (see \cite[Section 7]{sato-book}).

In the following theorem we give sufficient conditions for
$\lambda$-irreducibility of  Feller processes in terms of the
symbol.
\begin{theorem}\label{tm2.6}Let $\process{F}$ be a  Feller process  which satisfies conditions \textbf{(C1)}-\textbf{(C3)} and such that its corresponding symbol $q(x,\xi)$ satisfies the sector condition in (\ref{eq:1.1}) and
\be\label{eq:2.2}\int_{\R^{d}}\exp\left[-t\inf_{x\in\R^{d}}\rm{Re}\,\it{q}(x,\xi)\right]d\xi<\infty\ee
for all $t\geq0$. Then, $\process{F}$ possesses a density function
$p(t,x,y)$,  $t>0$ and $x,y\in\R^{d}.$ In addition, if
$\Phi_t(x,\xi)$ is real-valued and if there exists a function
$\underline{\Phi}_t(\xi)$ such that
$0<\underline{\Phi}_t(\xi)\leq\Phi_t(x,\xi)$ and
$\underline{\Phi}_{s+t}(\xi)\leq\underline{\Phi}_t(\xi)$  for all
$s,t\geq0$ and all $x,\xi\in\R^{d}$, then  $\process{F}$ is
$\lambda$-irreducible.
\end{theorem}
\begin{proof}
To prove the first claim, note that, by \cite[Theorem
2.7]{rene-wang-feller}, $\int_{\R^{d}}|\Phi_t(x,\xi)|d\xi<\infty$
for all $t>0$ and all $x\in\R^{d}$. Thus, the following functions
are well defined
$$p(t,x,y):=(2\pi)^{-d}\int_{\R^{d}}e^{-i\langle\xi,y\rangle}\Phi_t(x,\xi)d\xi$$
and $$\mathbb{P}^{x}(F_t\in B)=\int_{B-x}p(t,x,y)dy$$   for all
$t>0$, all $x,y\in\R^{d}$ and all $B\in\mathcal{B}(\R^{d})$.

To prove the second claim, again by \cite[Theorem
2.7]{rene-wang-feller}, for every $t>0$ we have
\begin{align*}|p(t,x,0)-p(t,x,y)|&=(2\pi)^{-d}\left|\int_{\R^{d}}\left(1-e^{-i\langle\xi,y\rangle}\right)\Phi_t(x,\xi)d\xi\right|\\&\leq
(2\pi)^{-d}\int_{\R^{d}}\left|1-e^{-i\langle\xi,y\rangle}\right|\exp\left[-\frac{t}{16}\inf_{x\in\R^{d}}\rm{Re}\,\it{q}(x,\xi)\right]d\xi.\end{align*}
Thus, by the dominated convergence theorem, for every $t_0>0$ the
continuity of the function $y\longmapsto p(t,x,y)$ at $0$ is
uniformly for all $t\geq t_0$ and all $x\in\R^{d}.$ Further, for
every $t_0>0$,
$$p(t,x,0)=(2\pi)^{-d}\int_{\R^{d}}\Phi_t(x,\xi)d\xi\geq(2\pi)^{-d}\int_{\R^{d}}\underline{\Phi}_{t_0+1}(\xi)d\xi>0$$
uniformly for all $t\in[t_0,t_0+1]$ and all $x\in\R^{d}$. According
to this, there exists $\varepsilon:=\varepsilon(t_0)>0$ such that
$p(t,x,y)>0$ for all $t\in[t_0,t_0+1]$, all $x\in\R^{d}$ and all
$|y|<\varepsilon$. Now, for any $n\in\N$, by the Chapman-Kolmogorov
equation, we have that  $p(t,x,y)>0$ for all $t\in[nt_0,n(t_0+1)]$,
all $x\in\R^{d}$ and all $|y|<n\varepsilon$. Finally, let
$B\in\mathcal{B}(\R^{d})$ such that $\lambda(B)>0$. Then, for  given
$t_0>0$ and $x\in\R^{d}$, there exists $n:=n(t_0,x)\in\N$, such that
$\lambda((B-x)\cap\{|y|<n\varepsilon\})>0$, where
$\varepsilon:=\varepsilon(t_0)>0$ is as above. Thus,
$$\mathbb{P}^{x}(F_t\in B)=\int_{B-x}p(t,x,y)dy\geq\int_{(B-x)\cap\{|y|<n\varepsilon\}}p(t,x,y)dy>0,$$
for all $t\in[nt_0,n(t_0+1)].$
\end{proof}
Note that  the condition in (\ref{eq:2.2}) follows from the
Hartman-Wintner condition in (\ref{eq:1.2}). Also, let us remark
that, in the spirit of Lemma \ref{l2.2}, we conjecture that a
symmetric Feller process $\process{F}$ with  symbol $q(x,\xi)$ which
satisfies conditions \textbf{(C1)}-\textbf{(C3)} also satisfies the
following uniform lower bound
$$\Phi_t(x,\xi)\geq\exp\left[-ct\sup_{z\in\R^{d}}q(z,\xi)\right]$$ for all
$t\geq0$, all $x,\xi\in\R^{d}$ and some constant $c>0.$ In
particular, under  the condition in (\ref{eq:2.2}), this implies the
$\lambda$-irreducibility of $\process{F}$.

 In
the following corollary we derive some conditions for the recurrence
and transience with respect to the dimension of the state space.
\begin{corollary}\label{c2.7}Let $\process{F}$ be a nice Feller process with  symbol $q(x,\xi)$.
\begin{enumerate}
  \item [(i)]If
$$\limsup_{|\xi|\longrightarrow0}\frac{\sup_{x\in\R^{d}}|q(x,\xi)|}{|\xi|^{\alpha}}<c$$
for
 some $\alpha>0$ and some $c<\infty$  and if
$d\leq\alpha$, then  the condition in (\ref{eq:1.3}) holds true.
  \item [(ii)]If
$$\liminf_{|\xi|\longrightarrow0}\frac{\inf_{x\in\R^{d}}\rm{Re}\,\it{q(x,\xi)}}{|\xi|^{\alpha}}>c$$
for
 some $\alpha>0$ and some $c>0$ and if
$d>\alpha$,
  then  the condition in (\ref{eq:1.4}) holds true.
\end{enumerate}
\end{corollary}
\begin{proof}\begin{enumerate}
               \item [(i)]For $r>0$ small enough and $d\leq\alpha$,  we have
$$\int_{\{|\xi|<r\}}\frac{d\xi}{\sup_{x\in\R^{d}}|q(x,\xi)|}\geq\frac{1}{c}\int_{\{|\xi|<r\}}\frac{d\xi}{|\xi|^{\rm{\alpha}}}=\frac{c_d}{c}\int_0^{r}\rho^{d-1-\alpha}d\rho=\infty,$$
where $c_d=d\pi^{\frac{d}{2}}\Gamma\left(\frac{d}{2}+1\right).$
               \item [(ii)] For $r>0$ small enough and $d>\alpha$, we have
$$\int_{\{|\xi|<r\}}\frac{d\xi}{\inf_{x\in\R^{d}}\rm{Re}\,\it{q}(x,\xi)}\leq\frac{1}{c}\int_{\{|\xi|<r\}}\frac{d\xi}{|\xi|^{\rm{\alpha}}}=\frac{c_d}{c}\int_0^{r}\rho^{d-1-\alpha}d\rho<\infty.$$
             \end{enumerate}
\end{proof}
As a direct consequence of Corollary \ref{c2.7} we get the following
conditions for the recurrence and transience with respect to the
dimension of the state space.
\begin{theorem} \label{tm2.8}Let  $\process{F}$ be a nice Feller process  with  symbol
$q(x,\xi)$. \begin{itemize}
          \item [(i)]
If $\process{F}$ is symmetric,  $d=1,2$ and
$$\sup_{x\in\R^{d}}\int_{\R^{d}}|y|^{2}\nu(x,dy)<\infty,$$  then $q(x,\xi)$ satisfies (\ref{eq:1.3}).
  \item [(ii)] If $d\geq3$ and
$$\liminf_{|\xi|\longrightarrow0}\frac{\sup_{c>0}\inf_{x\in\R^{d}}\left(\langle\xi,c(x)\xi\rangle+\int_{\{|y|\leq
c\}}\langle\xi,y\rangle^{\rm{2}}\nu(\it{x},dy)\right)}{|\xi|^{2}}>\rm{0},$$
then $q(x,\xi)$ satisfies (\ref{eq:1.4}).
\end{itemize}
\end{theorem}
\begin{proof}
\begin{enumerate}
  \item [(i)] The
claim easily follows from the facts that $1-\cos y\leq y^{2}$ for
all $y\in\R$,
$$\frac{\sup_{x\in\R^{d}}q(x,\xi)}{|\xi|^{2}}\leq
d\sup_{x\in\R^{d}}\max_{1\leq i,j\leq
d}|c_{ij}(x)|+\sup_{x\in\R^{d}}\int_{\R^{d}}|y|^{2}\nu(x,dy)$$ for
all $|\xi|$ small enough and Corollary \ref{c2.7} (i). Here we used
the fact that for  an arbitrary square matrix $A=(a_{ij})_{1\leq
i,j\leq d}$ and $\xi\in\R^{d}$, we have $|\langle
\xi,A\xi\rangle|\leq|\xi||A\xi|\leq d\max_{1\leq i,j\leq
d}|a_{ij}||\xi|^{2}.$
  \item [(ii)] The
claim is an immediate consequence of the facts that $1-\cos
y\geq\frac{1}{\pi}y^{2}$ for all $|y|\leq\frac{\pi}{2}$  and
$$\frac{\inf_{x\in\R^{d}}\rm{Re}\,\it{q(x,\xi)}}{|
\xi|^{2}}\geq
\frac{\inf_{x\in\R^{d}}\left(\langle\xi,c(x)\xi\rangle+\int_{\{|y|\leq
c\}}\langle\xi,y\rangle^{\rm{2}}\nu(\it{x},dy)\right)}{\pi|\xi|^{2}}$$
for all $c>0$ and all $|\xi|$ small enough and Corollary \ref{c2.7}
(ii).
\end{enumerate}
\end{proof}

As in the L\'evy process case, it is natural to expect that
$\rm{Re}\,\Phi_{\it{t}}(\it{x},\xi)=\rm{\Phi}_{\it{t}}(\it{x},\xi)$
if, and only if, $\process{F}$ is a  symmetric nice  Feller process.
The necessity easily follows from \cite[Theorem
2.1]{rene-wang-feller}. On the other hand, according to
\cite[Theorem 2.1 and Proposition 4.6]{rene-wang-feller}, the
sufficiency holds under the assumption that $C_c^{\infty}(\R^{d})$
is an operator core for the Feller generator
$(\mathcal{A},\mathcal{D}(\mathcal{A}))$, that is,
$\overline{\mathcal{A}|_{C_c^{\infty}(\R^{d})}}=\mathcal{A}$ on
$\mathcal{D}(\mathcal{A})$ (see also \cite[Theorem
1]{bjoern-rene-approx}). In the recurrence case, we require that
$\rm{Re}\,\Phi_{\it{t}}(\rm{0},\xi)\geq\rm{0}$ for all $t\geq0$ and
all $\xi\in\R^{d}$. Except in the symmetric L\'evy process case,
this assumption is trivially satisfied for nice Feller processes
which can be obtained by the symmetrization, which has been
introduced in \cite{rene-wang-feller}. Let $\process{F}$ be a nice
Feller process with  symbol $q(x,\xi)$ and let $\process{\bar{F}}$
be an independent copy of $\process{F}$. Let us define
$\tilde{F}_t:=2\bar{F}_0-\bar{F}_t.$ Then, by \cite[Theorem
2.1]{rene-wang-feller}, $\process{\tilde{F}}$ is a nice Feller
process with  symbol $\tilde{q}(x,\xi)=q(x,\xi)$ and
$\tilde{\Phi}_t(x,\xi)=\Phi_t(x,-\xi).$ Now, let  us define the
\emph{symmetrization} of $\process{F}$ by
$F^{s}_t:=(F_t+\tilde{F}_t)/2.$ Then, by \cite[Lemma
2.8]{rene-wang-feller}, $\process{F^{s}}$ is again a nice Feller
process with  symbol $q^{s}(x,\xi)=2\rm{Re}\,\it{q}(x,\xi/\rm{2})$
and
$\Phi^{s}_t=\Phi_t(x,\xi/2)\tilde{\Phi}_t(x,\xi/2)=|\Phi_t(x,\xi/2)|^{2}.$

Using the above symmetrization technique, we can consider the
recurrence and transience property of   Feller-Dynkin diffusions and
stable-like processes. Recall that a \emph{Feller-Dynkin diffusion}
is a Feller process with continuous sample paths satisfying
conditions (\textbf{C1}) and (\textbf{C3}). In particular,
Feller-Dynkin diffusions are determined by a symbol of the form
 $q(x,\xi)=- i\langle \xi,b(x)\rangle +
\frac{1}{2}\langle\xi,c(x)\xi\rangle$. Further,  it is easy to check
that the Feller generator $(\mathcal{A},\mathcal{D}_{\mathcal{A}})$
of Feller-Dynkin diffusions restricted to $C_c^{\infty}(\R^{d})$ is
a second-order elliptic operator
$$\mathcal{A}|_{C_c^{\infty}(\R^{d})}f(x)=\sum_{i=1}^{d}b_i(x)\frac{\partial f(x)}{\partial
x_i}+\frac{1}{2}\sum_{i,j=1}^{d}c_{ij}(x)\frac{\partial^{2}f(x)}{\partial
x_i \partial x_j}.$$ Hence, since Feller-Dynkin diffusions are
Feller process, $\mathcal{A}(C_c^{\infty}(\R^{d}))\subseteq
C_\infty(\R^{d})$, therefore $b(x)$ and $c(x)$ are continuous
functions. For further properties of Feller-Dynkin diffusions see
\cite{rogersI} and \cite{rogersII}.

\begin{theorem}\label{tm2.9} Let $\process{F}$ be a Feller-Dynkin
diffusion with  symbol $q(x,\xi)=- i\langle \xi,b(x)\rangle +
\frac{1}{2}\langle\xi,c(x)\xi\rangle$ which satisfies condition
(\textbf{C2}). Further, assume that
$\inf_{x\in\R^{d}}\langle\xi,c(x)\xi\rangle\geq c|\xi|^{2}$ for all
$\xi\in\R^{d}$ and some $c>0$   and that $b(x)$ and $c(x)$ are
H\"older continuous with the index $0<\beta\leq1$.
\begin{enumerate}
  \item [(i)] If $\process{F}$ is symmetric and $d=1,2$, then
  $\process{F}$ is recurrent.
  \item [(ii)] If $d\geq3$, then $\process{F}$ is transient.
\end{enumerate}
\end{theorem}
\begin{proof}First, let us remark that,  by \cite[Theorem A]{sheu}, $\process{F}$
possesses a strictly positive density function. In particular,
$\process{F}$ is $\lambda$-irreducible, that is, it satisfies
condition (\textbf{C4}). Hence, $\process{F}$ is a nice Feller
process.
\begin{enumerate}
               \item [(i)] Let $\process{\bar{F}}$ be a Feller-Dynkin diffusion
               with
symbol given by $\bar{q}(x,\xi)=\langle\xi,c(x)\xi\rangle$. The
existence (and uniqueness) of the process $\process{\bar{F}}$  is
given in \cite[Theorem 24.1]{rogersII}. Again, $\process{\bar{F}}$
is a symmetric nice Feller process. Now, by the symmetrization and
\cite[Theorem 24.1]{rogersII},
$\process{\bar{F}^{s}}\stackrel{\hbox{\scriptsize d}}{=}\process{F}$
and, in particular, $\Phi_t(x,\xi)\geq0$ for all $t\geq0$ and all
$x,\xi\in\R^{d}.$ Now, the claim easily follows from Theorem
\ref{tm2.8} (i).
               \item [(ii)] This is an immediate
               consequence of
               Theorems \ref{tm1.3} and \ref{tm2.8} (ii).
             \end{enumerate}
\end{proof}
Let $\alpha:\R^{d}\longrightarrow(0,2)$ and
$\gamma:\R^{d}\longrightarrow(0,\infty)$ be arbitrary bounded and
continuously differentiable functions with bounded derivatives, such
that
$0<\underline{\alpha}:=\inf_{x\in\R^{d}}\alpha(x)\leq\sup_{x\in\R^{d}}\alpha(x)=:\overline{\alpha}<2$
and $\inf_{x\in\R^{d}}\gamma(x)>0$. Under this assumptions, in
\cite{bass-stablelike} and
                                                         \cite[Theorem 3.3.]{rene-wang-feller}
                                                         it has been shown
                                                         that there
                                                         exists a
                                                         unique nice
                                                         Feller
                                                         process,
                                                         called a
                                                         \emph{stable-like
                                                         process},
                                                         determined
by a symbol of the form $p(x,\xi)=\gamma(x)|\xi|^{\alpha(x)}$. The
$\lambda$-irreducibility follows from Theorem \ref{tm2.6} and
\cite[Theorem 5.1]{vasili-stablelike}. Further, note that when
$\alpha(x)$ and $\gamma(x)$ are constant functions, then we deal
with a rotationally invariant stable L\'evy process. Now, as a
direct consequence of Theorems \ref{tm1.2} and \ref{tm1.3} and
Corollary \ref{c2.7}, we get the following conditions for recurrence
and transience of stable-like processes.
\begin{theorem}\label{tm2.10} Let $\process{F}$ be a stable-like process with  symbol given by $q(x,\xi)=\gamma(x)|\xi|^{\alpha(x)}.$\begin{enumerate}
\item [(i)]
                                                               If
                                                               $\process{F}$
                                                               is
                                                               one-dimensional
                                                               and
                                                               if
                                                               $\inf_{x\in\R}\alpha(x)\geq1$,
                                                               then
                                                               $\process{F}$
                                                               is
                                                               recurrent.
\item[(ii)] If
                                                               $\process{F}$
                                                               is
                                                               one-dimensional
                                                               and
                                                               if
                                                               $\sup_{x\in\R}\alpha(x)<1$,
                                                               then
                                                               $\process{F}$
                                                               is
                                                               transient.
                                                              \item[(iii)]
                                                               If
                                                               $d\geq2$,
then
                                                               $\process{F}$
                                                               is
                                                               transient.\end{enumerate}
\end{theorem}
\begin{proof}
Let $\process{\bar{F}}$ be a stable-like process determined by a
symbol of the form
$$\bar{q}(x,\xi)=2^{\alpha(x)-1}\gamma(x)|\xi|^{\alpha(x)}.$$ Then, by
the symmetrization,
$\process{\bar{F}^{s}}\stackrel{\hbox{\scriptsize d}}{=}\process{F}$
and, in particular, $\Phi_t(x,\xi)\geq0$ for all $t\geq0$ and all
$x,\xi\in\R^{d}.$ Now, the desired results easily follow from
Theorems \ref{tm1.2} and \ref{tm1.3} and Corollary \ref{c2.7}.
\end{proof}
Let us remark that the recurrence and transience property of
stable-like processes has been studied extensively in  the
literature (see \cite{bjoern-overshoot}, \cite{franke-periodic},
\cite{sandric-periodic}, \cite{sandric-ergodic}, \cite{sandric-spa},
\cite{sandric-rectrans},
 \cite{rene-wang-feller}).

In what follows, we briefly discuss the recurrence and transience
property of symmetric nice Feller processes  obtained by variable
order subordination (see also \cite{rene-wang-feller}). Let
$q:\R^{d}\longrightarrow\R$ be a continuous and negative definite
function such that $q(0) = 0$ (that is, $q(\xi)$ is the symbol of
some symmetric L\'evy process). Further, let $f :\R^{d}\times[0,
\infty)\longrightarrow[0,\infty)$ be a measurable function such that
$\sup_{x\in\R^{d}} f(x, t)\leq c(1 + t)$ for some $c\geq0$ and all
$t\in[0,\infty)$, and for fixed $x\in\R^{d}$ the function
$t\longrightarrow f(x, t)$ is a Bernstein function with $f(x, 0) =
0$. Bernstein functions are the characteristic Laplace exponents of
subordinators (L\'evy processes with nondecreasing sample paths).
For more on Bernstein functions we refer the readers to the
monograph   \cite{Bernstein}. Now, since $q(\xi)\geq0$ for all
$\xi\in\R^{d}$, the function
$$\bar{q}(x,\xi) := f(x,q(\xi))$$ is well
defined  and, according to \cite[Theorem 5.2]{Bernstein} and
\cite[Theorem 30.1]{sato-book},  $\xi\longmapsto \bar{q}(x,\xi)$  is
a continuous and negative definite function  satisfying conditions
(\textbf{C2}) and (\textbf{C3}). Hence, $\bar{q}(x,\xi)$ is possibly
the symbol of some  symmetric Feller process. Of special interest is
the case when $f(x,t)=t^{\alpha(x)},$ where
$\alpha:\R^{d}\longmapsto [0,1]$, that is,  $\bar{q}(x,\xi)$
describes variable order subordination. For sufficient conditions on
the symbol $q(\xi)$ and function $\alpha(x)$ such that
$\bar{q}(x,\xi)$ is the symbol of some Feller process see
\cite{jacob} and \cite{hoh} and  the references therein. Now, let
$0\leq\underline{\alpha}:=\inf_{x\in\R^{d}}\alpha(x)\leq\sup_{x\in\R^{d}}\alpha(x)=:\overline{\alpha}\leq1$.
Then, since the symbol $q(\xi)$ is continuous and $q(0)=0$, there
exists $r>0$ small enough, such that $q(\xi)\leq1$ for all
$|\xi|<r.$ In particular, the conditions in (\ref{eq:1.3}) and
(\ref{eq:1.4}) hold true if
$$\int_{\{|\xi|<r\}}\frac{d\xi}{\left(q(\xi)\right)^{\underline{\alpha}}}=\infty\quad\textrm{and}\quad
\int_{\{|\xi|<r\}}\frac{d\xi}{\left(q(\xi)\right)^{\overline{\alpha}}}<\infty,$$
respectively (see also \cite[Corollary 3.2]{rene-wang-feller}). Note
that when $q(\xi)$ is the symbol of a  Brownian motion, then by
variable order subordination we get a stable-like process.

If we know the distribution of $\process{F}$, in order to prove the
recurrence of $\process{F}$, the assumption
$\rm{Re}\,\Phi_{\it{t}}(0,\xi)\geq\rm{0}$ for all $t\geq0$ and all
$\xi\in\R^{d}$ can be replaced by the following assumptions.
\begin{proposition}Let $\process{F}$ be a nice Feller process with  symbol $q(x,\xi)$ which
satisfies  the condition in (\ref{eq:1.3}). If there exists
$x_0\in\R^{d}$ such that for every $a>0$ there exist $b>0$,
$\varepsilon>0$ and $t_0\geq0$, such that
$$\mathbb{P}^{x_0}(F_t\in B_a(x_0))\geq \varepsilon\sup_{y\in\R^{d}}\mathbb{P}^{x_0}(F_t\in B_b(2x_0-y))$$ for all $t\geq t_0$, then
$\process{F}$ is recurrent. Here,  $B_r(x):=\{z:|z-x|<r\}$ denotes
the open ball of radius $r>0$ around $x\in\R^{d}$. In addition, if
$q(x,\xi)$ satisfies  the condition in (\ref{eq:2.2}), then
$\process{F}$ is recurrent if there exists $x_0\in\R^{d}$ such that
for every $a>0$ there exist $\varepsilon>0$ and $t_0\geq0$, such
that
$$\mathbb{P}^{x_0}(F_t\in B_a(x_0))\geq \varepsilon\int_{\R^{d}}\exp\left[-t\inf_{x\in\R^{d}}\rm{Re}\,\it{q}(x,\xi)\right]d\xi $$ for all
$t\geq t_0$.
\end{proposition}
\begin{proof} Let $\process{F^{s}}$ be the symmetrization of
$\process{F}$. Then, by assumption, $\process{F^{s}}$ is recurrent.
Next, in order to prove the recurrence of $\process{F}$, by
Proposition \ref{p2.1} (iv), it suffices to show that there exists
$x\in\R^{d}$ such that for  every $a>0$ we have
$$\int_0^{\infty} \mathbb{P}^{x}(F_t\in B_a(x))=\infty.$$ Let $a>0$ be arbitrary
and let $x_0\in\R^{d}$,  $b>0$, $\varepsilon>0$ and $t_0\geq0$ be as
above. Then, for $t\geq t_0$ we have
\begin{align*}\mathbb{P}^{x_0}(F^{s}_t\in
B_{b/2}(x_0))=\int_{\R^{d}}\mathbb{P}^{x_0}(F_t\in
B_b(2x_0-y))\mathbb{P}^{x_0}(F_t\in dy)\leq
\frac{\mathbb{P}^{x_0}(F_t\in B_a(x_0))}{\varepsilon}.
\end{align*}

 To prove the second part, note that, by Theorem \ref{tm2.6},
 $\mathbb{P}^{x}(F_t\in dy)=p(t,x,y)dy$, for $t>0$ and $x,y\in\R^{d}.$ Thus, by
 \cite[Theorem 1.1]{rene-wang-feller}, we have
$$\mathbb{P}^{x_0}(F_t\in B_b(2x_0-y))=\int_{B_{b}(2x_0-y)-x_0}p(t,x_0,z)dz\leq\frac{\lambda(B_b(0))}{(4\pi)^{d}}\int_{\R^{d}}\exp\left[-\frac{t}{16}\inf_{x\in\R^{d}}\rm{Re}\,\it{q}(x,\xi)\right]d\xi$$
for all $t>0$ and all $y\in\R^{d}$, where $x_0\in\R^{d}$ and $b>0$
are as above.
\end{proof}

 In Section \ref{s1}, we recalled the notion of stable L\'evy processes and the
indices of stability. The concept of the indices of stability  can
be generalized to  general L\'evy process through the so-called
Pruitt indices (see \cite{pruitt}). The \emph{Pruitt indices}, for a
nice Feller process $\process{F}$ with  symbol $q(x,\xi)$, are
defined in the following way
\begin{align*}\beta&:=\sup\left\{\alpha\geq0:\lim_{|\xi|\longrightarrow0}\frac{\sup_{x\in\R^{d}}|q(x,\xi)|}{|\xi|^{\alpha}}=0\right\}\\
\delta&:=\sup\left\{\alpha\geq0:\lim_{|\xi|\longrightarrow0}\frac{\inf_{x\in\R^{d}}\rm{Re}\,\it{q}(\it{x},\xi)}{|\xi|^{\alpha}}=\rm{0}\right\}\end{align*}
(see \cite{rene-holder}). Note that $0\leq\beta\leq\delta$,
$\beta\leq2$ and in the case of a  stable-like process with symbol
given by $q(x,\xi)=\gamma(x)|\xi|^{\alpha(x)}$, we have
$\beta=\underline{\alpha}$ and $\delta=\overline{\alpha}.$  Now, we
 generalize Theorem
  \ref{tm2.10} in terms of the Pruitt indices.
\begin{theorem}\label{tm2.12}Let $\process{F}$ be a   nice Feller process with  symbol $q(x,\xi)$.
\begin{enumerate}
  \item [(i)] If  $d=1$ and $\beta>1$, then  the condition in (\ref{eq:1.3}) holds true.
\item [(ii)] If  $q(x,\xi)$
satisfies the condition in  (\ref{eq:1.4}), then $\delta\leq d$.
\end{enumerate}
\end{theorem}
\begin{proof} \begin{enumerate}
                \item [(i)] Let $1\leq\alpha<\beta$ be arbitrary. Then, by the definition of
$\beta$,
$$\lim_{|\xi|\longrightarrow0}\frac{\sup_{x\in\R}|q(x,\xi)|}{|\xi|^{\alpha}}=0.$$
Now, the claim easily follows from Corollary \ref{c2.7} (i).
\item[(ii)]
   Let us
assume that this is not the case.  Then, for all
$d\leq\alpha<\delta$, by the definition of $\delta$, we have
$$\lim_{|\xi|\longrightarrow0}\frac{\inf_{x\in\R^{d}}\rm{Re}\,\it{q}(x,\xi)}{|\xi|^{\alpha}}=\rm{0}.$$
By Corollary \ref{c2.7} (i), this yields that
$$\int_{\{|\xi|<r\}}\frac{d\xi}{\inf_{x\in\R^{d}}\rm{Re}\,\it{q}(x,\xi)}\geq\int_{\{|\xi|<r\}}\frac{d\xi}{|\xi|^{\alpha}}=\infty$$
for all $r>0$ small enough.
\end{enumerate}
\end{proof}
Let us remark that, according to Theorem \ref{tm2.10}, in the
dimension $d\geq2$ even under  the condition in (\ref{eq:1.4}) we
can have $\delta\geq1$. Similarly, the recurrence of a
one-dimensional symmetric nice Feller process, in general, does not
imply that $\beta\geq1$. To see this, first recall that for two
symmetric measures $\mu(dy)$ and $\bar{\mu}(dy)$ on
$\mathcal{B}(\R)$ which are finite outside of any neighborhood of
the origin, we say that $\mu(dy)$ has a bigger tail than
$\bar{\mu}(dy)$ if there exists $y_0>0$ such that
$\mu(y,\infty)\geq\bar{\mu}(y,\infty)$ for all $y\geq y_0$. Now, by
\cite[Theorem 38.4]{sato-book}, if $\bar{\nu}(dy)$ is the L\'evy
measure of a transient one-dimensional symmetric  L\'evy process
$\process{\bar{L}}$, then there exists a recurrent one-dimensional
symmetric L\'evy process $\process{L}$ with  L\'evy measure
$\nu(dy)$ having a bigger tail then $\bar{\nu}(dy)$. Further, by
Fubini's theorem, for any $\alpha>0$ we have
\begin{align*}\int_{\{y>y_0\}}y^{\alpha}\nu(dy)&=\int_{\{y>y_0\}}\int_0^{y}\alpha
z^{\alpha-1}dz\,\nu(dy)\\&=y_0^{\alpha}\nu(y_0,\infty)+\alpha\int_{\{z>y_0\}}z^{\alpha-1}\nu(z,\infty)dz\\
&\geq
y_0^{\alpha}\nu(y_0,\infty)+\alpha\int_{\{z>y_0\}}z^{\alpha-1}\bar{\nu}(z,\infty)dz.\end{align*}
Hence, if $\int_{\{y>1\}}y^{\alpha}\bar{\nu}(dy)=\infty$, then
$\int_{\{y>1\}}y^{\alpha}\nu(dy)=\infty.$ Therefore, by
\cite[Prposition 5.4]{rene-holder}, the recurrence of $\process{L}$
does not imply that $\beta\geq1$. Similarly, by \cite[Theorem
38.4]{sato-book} and \cite[Prposition 5.4]{rene-holder}, in the
one-dimensional symmetric case, $\delta\leq1$ does not automatically
imply transience. By assuming certain regularities (convexity and
concavity) on  symbol $q(x,\xi)$, we get the converse of Theorem
\ref{tm2.12} (see  \cite[Theorem 38.2]{sato-book} for the L\'evy
process case).

 \begin{theorem}\label{tm2.13}Let $\process{F}$ be a  nice Feller process with  symbol
 $q(x,\xi)$.
\begin{enumerate}
\item [(i)]If $d=1$,   the function $\xi\longmapsto\sup_{x\in\R}|\it{q}(\it{x},\xi)|$ is radial
and convex on some neighborhood around the origin  and  $q(x,\xi)$
satisfies  the condition in (\ref{eq:1.3}),  then $\beta\geq1$.
\item[(ii)] If  the function
$\xi\longmapsto\sup_{x\in\R^{d}}|\it{q}(\it{x},\xi)|$ is radial and
concave on some neighborhood around the origin,  then $\beta\leq1$.
\item [(iii)]If the function $\xi\longmapsto\inf_{x\in\R^{d}}\rm{Re}\,\it{q}(\it{x},\xi)$ is radial
and concave  on some neighborhood around the origin, $d\geq2$ and
 $\delta<d$, then $q(x,\xi)$ satisfies the condition
in (\ref{eq:1.4}).
\item[(iv)]If the function
$\xi\longmapsto\inf_{x\in\R^{d}}\rm{Re}\,\it{p}(\it{x},\xi)$ is
radial and convex on some neighborhood around the origin, then
$\delta\geq1$.
\end{enumerate}
 \end{theorem}
\begin{proof}
\begin{enumerate}
\item [(i)] We show that
$$\lim_{|\xi|\longrightarrow0}\frac{\sup_{x\in\R}|q(x,\xi)|}{|\xi|^{\alpha}}=0$$
for all $\alpha<1$.  Let us assume that this is not the case.
                Then, there exists $\alpha_0<1$ such that
                $$\limsup_{|\xi|\longrightarrow0}\frac{\sup_{x\in\R}|q(x,\xi)|}{|\xi|^{\alpha_0}}>c$$ for some $c>0$.
                Hence, there exists a sequence $\{\xi_n\}_{n\in\N}\subseteq\R$
                such that $\lim_{n\longrightarrow\infty}|\xi_n|=0$
                and $$\lim_{n\longrightarrow\infty}\frac{\sup_{x\in\R}|\it{q(x,\xi_n)|}}{|\xi_n|^{\alpha_0}}>c.$$ Thus, $\sup_{x\in\R}|\it{q(x,\xi_n)|}\geq c|\xi_n|^{\alpha_{\rm{0}}}$ for all
                $n\in\N$ large enough. Now, because of the radial symmetry and  convexity assumptions,
$\sup_{x\in\R}              |\it{q(x,\xi)|}\geq
c|\xi|^{\alpha_{\rm{0}}}$ for all
                $|\xi|$ small enough. Indeed, if this was not the case, then
                there would exist a sequence $\{\bar{\xi}_n\}_{n\in\N}\subseteq\R$
                such that $\lim_{n\longrightarrow\infty}|\bar{\xi}_n|=0$
                and $\sup_{x\in\R^{d}}|
                \it{q(x,\bar{\xi}_n)|}<c|\bar{\xi}_n|^{\alpha_{\rm{0}}}$ for
                all $n\in\N$. Now, let $n,m\in\N$ be such that $|\xi_n|<|\bar{\xi}_m|$ and $\sup_{x\in\R}|\it{q(x,\xi_n)|}\geq c|\xi_n|^{\alpha_{\rm{0}}}$.
                Then, for adequately chosen $t\in[0,1]$, we have
                \begin{align*}c|\xi_n|^{\alpha_0}\leq\sup_{x\in\R}|
                \it{q(x,\xi_n)|}=\sup_{x\in\R}|
                \it{q(x,t\bar{\xi}_m|}\leq \it{t}\sup_{x\in\R}|
                \it{q(x,\bar{\xi}_m)}|<
                ct|\bar{\xi}_m|^{\alpha_{\rm{0}}}\leq ct^{\alpha_{\rm{0}}}|\bar{\xi}_m|^{\alpha_{\rm{0}}}=c|\xi_n|^{\alpha_{\rm{0}}},\end{align*}where
                in the third step we took into account the convexity
                property.
 Hence, $\sup_{x\in\R}|
                \it{q(x,\xi)|}\geq c|\xi|^{\alpha_{\rm{0}}}$ for all
                $|\xi|$ small enough. But, according to Proposition \ref{p2.4}, this is in contradiction with
                 the condition in (\ref{eq:1.3}).
\item[(ii)] Let $\varepsilon>0$ be such that
$|\xi|\longmapsto\sup_{x\in\R^{d}}|{q}(\it{x},\xi)|$ is concave on
$[0,\varepsilon)$. Now, for  all $\alpha\geq1$, we have
\begin{align*}\liminf_{|\xi|\longrightarrow0}\frac{\sup_{x\in\R^{d}}
                |\it{q(x,\xi)|}}{|\xi|^{\alpha}}&=\liminf_{|\xi|\longrightarrow0}\frac{\sup_{x\in\R^{d}}
                \left|\it{q\left(x,\frac{\rm{2}|\xi|}{\varepsilon}\frac{\varepsilon\xi}{\rm{2}|\xi|}\right)}\right|}{|\xi|^{\alpha}}\\&\geq \liminf_{|\xi|\longrightarrow0}\frac{2|\xi|^{1-\alpha}}{\varepsilon}\sup_{x\in\R^{d}}
                \left|\it{q\left(x,\frac{\varepsilon\xi}{\rm{2}|\xi|}\right)}\right|>0,\end{align*}
where in the second step we applied the concavity property.
                Now, the desired result follows from the definition
                of the index $\beta$.
 \item [(iii)] Let  $\max\{1,\delta\}<\alpha< d$ be arbitrary.
                By the definition of $\delta$, we have
                $$\limsup_{|\xi|\longrightarrow0}\frac{\inf_{x\in\R^{d}}\rm{Re}\,
                \it{q(x,\xi)}}{|\xi|^{\alpha}}>c$$ for some $c>0$.
                Thus, there exists a sequence $\{\xi_n\}_{n\in\N}\subseteq\R^{d}$
                such that $\lim_{n\longrightarrow\infty}|\xi_n|=0$
                and
                $$\lim_{n\longrightarrow\infty}\frac{\inf_{x\in\R^{d}}\rm{Re}\,
                \it{q(x,\xi_n)}}{|\xi_n|^{\alpha}}>c.$$ In particular,
                $\inf_{x\in\R^{d}}\rm{Re}\,
                \it{q(x,\xi_n)}\geq c|\xi_n|^{\alpha}$ for all
                $n\in\N$ large enough. Actually, because of the radial symmetry and concavity assumptions,
$\inf_{x\in\R^{d}}\rm{Re}\,
                \it{q(x,\xi)}\geq c|\xi|^{\alpha}$ for all
                $|\xi|$ small enough. Indeed,  for all $n\in\N$ large enough and  all $t\in[0,1]$  we have
        \begin{align*}\inf_{x\in\R^{d}}\rm{Re}\,\it{q(x,t\xi_n)}\geq
                t\inf_{x\in\R^{d}}\rm{Re}\,\it{q(x,\xi_n)}\geq ct|\xi_n|^{\alpha}\geq ct^{\alpha}|\xi_n|^{\alpha}=c|t\xi_n|^{\alpha},\end{align*}
where in the first step we took into account the concavity property.
Now, the
                claim is a direct consequence of Corollary \ref{c2.7}
                (ii).
\item[(iv)]  Let $\varepsilon>0$ be such that
$|\xi|\longmapsto\inf_{x\in\R^{d}}\rm{Re}\,\it{q}(\it{x},\xi)$ is
convex on $[0,\varepsilon)$ and let   $\alpha<1$ be arbitrary. Then,
we have
\begin{align*}\limsup_{|\xi|\longrightarrow0}\frac{\inf_{x\in\R^{d}}\rm{Re}\,
                \it{q(x,\xi)}}{|\xi|^{\alpha}}&=\limsup_{|\xi|\longrightarrow0}\frac{\inf_{x\in\R^{d}}\rm{Re}\,
                \it{q\left(x,\frac{\rm{2}|\xi|}{\varepsilon}\frac{\varepsilon\xi}{\rm{2}|\xi|}\right)}}{|\xi|^{\alpha}}\\&\leq
                \limsup_{|\xi|\longrightarrow0}\frac{2\inf_{x\in\R^{d}}\rm{Re}\,
                \it{q\left(x,\frac{\varepsilon\xi}{\rm{2}|\xi|}\right)}}{\varepsilon|\xi|^{\alpha-1}}=0,\end{align*}
where in the second step we employed the convexity property.
                Now, the desired result follows from the definition
                of the index  $\delta.$
\end{enumerate}
\end{proof}
Also, let us remark that the conclusions of Theorem \ref{tm2.13} can
be easily obtained if instead of the convexity and concavity and
radial symmetry assumptions on the functions
$\xi\longmapsto\sup_{x\in\R^{d}}|q(x,\xi)|$ and
$\xi\longmapsto\inf_{x\in\R^{d}}\rm{Re}\,\it{q}(\it{x},\xi)$ we
assume that
$$c^{-1}|\xi|^{\alpha}\leq\sup_{x\in\R^{d}}|q(x,\xi)|\leq
c|\xi|^{\beta}\quad\textrm{and}\quad
c^{-1}|\xi|^{\alpha}\leq\inf_{x\in\R^{d}}\rm{Re}\,\it{q}(\it{x},\xi)|\leq
c|\xi|^{\beta}$$ for all $|\xi|$ small enough and some adequate
$0<\beta\leq\alpha<\infty$ and $c>0$.

\section{Recurrence and transience of one-dimensional  symmetric nice Feller
processes}\label{s3}

\ \ \ \ In this section, we consider the recurrence and transience
property of one-dimensional symmetric nice Feller processes. Note
that in this case Proposition \ref{p2.4}  holds true, that is, the
condition in (\ref{eq:1.3}) does not depend on $r>0$. On the other
hand, recall that if the condition in (\ref{eq:1.4}) holds for some
$r_0>0$, then it also holds for all $r>r_0$. In situations where we
need complete independence of $r>0$, we assume the subadditivity of
the function
$\xi\longmapsto\inf_{x\in\R}\sqrt{\it{q}\left(x,\xi\right)}$ (see
Proposition \ref{p2.4}).

First, we study perturbations of symbols which do not affect the
recurrence and transience property of the underlying Feller process.
\begin{theorem}\label{tm3.1} Let $\process{F^{1}}$ and $\process{F^{2}}$ be
one-dimensional symmetric nice Feller processes with  symbols
$q_1(x,\xi)$ and $q_2(x,\xi)$ and  L\'evy measures $\nu_1(x,dy)$ and
$\nu_2(x,dy),$ respectively. If $q_1(x,\xi)$ satisfies
(\ref{eq:1.3}) and
\begin{align}\label{eq:3.1}\sup_{x\in\R}\int_0^{\infty}
y^{2}|\nu_1(x,dy)-\nu_2(x,dy)|<\infty,\end{align} then $q_2(x,\xi)$
also satisfies  (\ref{eq:1.3}). Further, if $q_1(x,\xi)$ satisfies
(\ref{eq:1.4}),
\be\label{eq:3.2}\lim_{\xi\longrightarrow0}\frac{\inf_{x\in\R}q_1(x,\xi)}{\xi^{2}}=\infty\ee
and  (\ref{eq:3.1}), then $q_2(x,\xi)$ also satisfies (\ref{eq:1.4})
and (\ref{eq:3.2}). Here, $|\mu(dy)|$ denotes the total variation
measure of the signed measure $\mu(dy)$.
\end{theorem}
\begin{proof} Let
$$q_1(x,\xi)=\frac{1}{2}c_1(x)\xi^{2}+\int_\R(1-\cos\xi
y)\nu_1(x,dy)\quad \textrm{and}\quad
q_2(x,\xi)=\frac{1}{2}c_2(x)\xi^{2}+\int_\R(1-\cos\xi
y)\nu_2(x,dy)$$ be the symbols of $\process{F^{1}}$ and
$\process{F^{2}}$, respectively. First, let us prove the recurrence
case. Note that (\ref{eq:3.1}) implies that
$$\sup_{x\in\R}\int_0^{\infty} y^{2}\nu_1(x,dy)<\infty\quad \textrm{if, and only
if,}\quad \sup_{x\in\R}\int_0^{\infty} y^{2}\nu_2(x,dy)<\infty.$$
Indeed, we have
\begin{align*}\int_0^{\infty} y^{2}\nu_1(x,dy)&=\int_0^{\infty}
y^{2}|\nu_1(x,dy)-\nu_2(x,dy)+\nu_2(x,dy)|\\ &\leq\int_0^{\infty}
y^{2}\nu_2(x,dy)+\int_0^{\infty}
y^{2}|\nu_1(x,dy)-\nu_2(x,dy)|\end{align*} and similarly
\begin{align*}\int_0^{\infty}y^{2}\nu_2(x,dy)&\leq\int_0^{\infty}
y^{2}\nu_1(x,dy)+\int_0^{\infty}
y^{2}|\nu_1(x,dy)-\nu_2(x,dy)|.\end{align*} Now, in the case when
$\sup_{x\in\R}\int_0^{\infty} y^{2}\nu_1(x,dy)<\infty$, the claim
easily follows from Theorem \ref{tm2.8} (i). Suppose that
$\sup_{x\in\R}\int_0^{\infty} y^{2}\nu_1(x,dy)=\infty$. Then, by
Fatou's lemma, we have
$$\liminf_{\xi\longrightarrow0}\sup_{x\in\R}\int_0^{\infty}\frac{1-\cos\xi
y}{\xi^{2}}\nu_{1}(x,dy)\geq\liminf_{\xi\longrightarrow0}\int_0^{\infty}\frac{1-\cos\xi
y}{\xi^{2}}\nu_{1}(x,dy)=\frac{1}{2}\int_0^{\infty}y^{2}\nu_{1}(x,dy).$$
Hence,\begin{align}\label{eq:3.3}\lim_{\xi\longrightarrow0}\frac{\sup_{x\in\R}q_1(x,\xi)}{\xi^{2}}=\infty.\end{align}
Next, we have
\begin{align}\label{eq:3.4}&|\sup_{x\in\R}q_1(x,\xi)-\sup_{x\in\R}q_2(x,\xi)|\nonumber\\&\leq\sup_{x\in\R}|q_1(x,\xi)-q_2(x,\xi)|\nonumber\\&\leq
\frac{1}{2}\sup_{x\in\R}|c_1(x)-c_2(x)|\xi^{2}+2\sup_{x\in\R}\left|\int_0^{\infty}
(1-\cos\xi y)\nu_1(x,dy)-\int_0^{\infty} (1-\cos\xi
y)\nu_2(x,dy)\right|\nonumber\\&\leq
\frac{1}{2}\sup_{x\in\R}|c_1(x)-c_2(x)|\xi^{2}+2\sup_{x\in\R}\int_0^{\infty}
(1-\cos\xi y)|\nu_1(x,dy)-\nu_2(x,dy)|\nonumber\\&\leq
\frac{1}{2}\sup_{x\in\R}|c_1(x)-c_2(x)|\xi^{2}+2\xi^{2}\sup_{x\in\R}\int_0^{\infty}
y^{2}|\nu_1(x,dy)-\nu_2(x,dy)|\nonumber\\&\leq
\left(\frac{1}{2}\sup_{x\in\R}|c_1(x)-c_2(x)|+2\sup_{x\in\R}\int_0^{\infty}
y^{2}|\nu_1(x,dy)-\nu_2(x,dy)|\right)\xi^{2},
\end{align} where in the fourth step we used the fact that $1-\cos y\leq y^{2}$
for all $y\in\R.$ Finally, by (\ref{eq:3.3}) and (\ref{eq:3.4}), we
have
$$\lim_{\xi\longrightarrow0}\frac{\sup_{x\in\R}q_2(x,\xi)}{\sup_{x\in\R}q_2(x,\xi)}=1+\lim_{\xi\longrightarrow0}\frac{\sup_{x\in\R}q_2(x,\xi)-\sup_{x\in\R}q_1(x,\xi)}{\sup_{x\in\R}q_1(x,\xi)}=1,$$
which together with  Proposition \ref{p2.4}  proves the claim.

 In the
transience case, we proceed in the similar way. We have
\begin{align}\label{eq:3.5}&|\inf_{x\in\R}q_1(x,\xi)-\inf_{x\in\R}q_2(x,\xi)|\nonumber\\&\leq\sup_{x\in\R}|q_1(x,\xi)-q_2(x,\xi)|\nonumber\\&\leq \left(\frac{1}{2}\sup_{x\in\R}|c_1(x)-c_2(x)|+2\sup_{x\in\R}\int_0^{\infty}
y^{2}|\nu_1(x,dy)-\nu_2(x,dy)|\right)\xi^{2}.
\end{align} Hence, by (\ref{eq:3.2}) and (\ref{eq:3.5}),  we have
$$\lim_{\xi\longrightarrow0}\frac{\inf_{x\in\R}q_2(x,\xi)}{\inf_{x\in\R}q_1(x,\xi)}=1+\lim_{\xi\longrightarrow0}\frac{\inf_{x\in\R}q_2(x,\xi)-\inf_{x\in\R}q_1(x,\xi)}{\inf_{x\in\R}q_1(x,\xi)}=1$$
and
$$\lim_{\xi\longrightarrow0}\frac{\inf_{x\in\R}q_2(x,\xi)}{\xi^{2}}=\infty.$$
Now, by applying Proposition \ref{p2.4}, the claim follows.
\end{proof}
Let us remark that it is easy to see that the condition in
(\ref{eq:3.2}) can be relaxed to the following condition
$$\liminf_{\xi\longrightarrow0}\frac{\inf_{x\in\R}q_1(x,\xi)}{\xi^{2}}>\frac{1}{2}\sup_{x\in\R}|c_1(x)-c_2(x)|+2\sup_{x\in\R}\int_0^{\infty}
y^{2}|\nu_1(x,dy)-\nu_2(x,dy)|.$$ Theorem \ref{tm3.1} essentially
says that, in the one-dimensional symmetric case, the recurrence and
transience property of nice Feller processes depends only on big
jumps.   A situation where the perturbation condition in
(\ref{eq:3.1}) easily holds true is given in the following
proposition.
\begin{proposition}\label{p3.2}
 Let $\process{F^{1}}$ and $\process{F^{2}}$ be
one-dimensional symmetric nice Feller processes with  L\'evy
measures $\nu_1(x,dy)$ and $\nu_2(x,dy)$, respectively. If there
exists $y_0>0$ such that\\ $\nu_1(x,(y,\infty))=\nu_2(x,(y,\infty))$
for all $x\in\R$ and all $y\geq y_0$, then  the condition in
(\ref{eq:3.1}) holds true.
\end{proposition}
Now, as a simple consequence of Theorem \ref{tm3.1} and Proposition
\ref{p3.2} we can generalize Theorem \ref{tm2.10} (see also
\cite[Theorem 4.6]{bjoern-overshoot}).
\begin{corollary}\label{c3.3} Let $\process{F}$ be a one-dimensional stable-like process with  symbol $q(x,\xi)=\gamma(x)|\xi|^{\alpha(x)}.$\begin{enumerate}
\item [(i)]
                                                               If
                                                               $\liminf_{|x|\longrightarrow\infty}\alpha(x)\geq1$,
                                                               then
                                                               $\process{F}$
                                                               is
                                                               recurrent.
\item[(ii)] If
                                                               $\limsup_{|x|\longrightarrow\infty}\alpha(x)<1$,
                                                               then
                                                               $\process{F}$
                                                               is
                                                               transient.
                                                              \end{enumerate}
\end{corollary}
Let us remark that by allowing
$\liminf_{|x|\longrightarrow\infty}\alpha(x)=1,$ the above corollary
also generalizes \cite[Theorem 1.3]{sandric-rectrans}, \cite[Theorem
1.3]{sandric-ergodic} and \cite[Theorem 1.1]{sandric-spa}. In the
following theorem, we slightly generalize Proposition \ref{p3.2}.
\begin{theorem} Let $\process{F^{1}}$ and $\process{F^{2}}$ be one-dimensional symmetric nice Feller
processes with symbols $q_1(x,\xi)$ and $q_2(x,\xi)$ and L\'evy
measures $\nu_1(x,dy)$ and $\nu_2(x,dy)$, respectively. Further,
assume that there exists a compact set $C\subseteq\R$ such that
$\nu_1(x,B\cap C^{c})\geq\nu_2(x,B\cap C^{c})$ for all $x\in\R$ and
all $B\in\mathcal{B}(\R)$. If $q_1(x,\xi)$ satisfies (\ref{eq:1.3}),
then $q_2(x,\xi)$ also satisfies (\ref{eq:1.3}). Next, if
$q_2(x,\xi)$ satisfies  (\ref{eq:1.4}) and (\ref{eq:3.2}), then
$q_1(x,\xi)$ also satisfies (\ref{eq:1.4}) and (\ref{eq:3.2}).
\end{theorem}
\begin{proof}Let
$$q_1(x,\xi)=\frac{1}{2}c_1(x)\xi^{2}+\int_{\R}(1-\cos\xi
y)\nu_1(x,dy)\quad \textrm{and}\quad
q_2(x,\xi)=\frac{1}{2}c_2(x)\xi^{2}+\int_{\R}(1-\cos\xi
y)\nu_2(x,dy)$$ be the symbols of $\process{F^{1}}$ and
$\process{F^{2}}$, respectively. First, let us prove the recurrence
case.  Let $m>0$ be so large  that $C\subseteq[-m,m]$. We have
\begin{align*}q_2(x,\xi)&=\frac{1}{2}c_2(x)\xi^{2}+\int_{\R}(1-\cos\xi
y)\nu_2(x,dy)\\&=\frac{1}{2}c_2(x)\xi^{2}+\int_{[-m,m]}(1-\cos\xi
y)\nu_2(x,dy)+\int_{[-m,m]^{^{c}}}(1-\cos\xi y)\nu_2(x,dy)\\&\leq
\frac{1}{2}c_2(x)\xi^{2}+\int_{[-m,m]}(1-\cos\xi
y)\nu_2(x,dy)+\int_{[-m,m]^{c}}(1-\cos\xi
y)\nu_1(x,dy)\\&\leq\frac{1}{2}\sup_{x\in\R}c_2(x)\xi^{2}+\xi^{2}\sup_{x\in\R}\int_{[-m,m]}y^{2}\nu_2(x,dy)+\int_{[-m,m]^{c}}(1-\cos\xi
y)\nu_1(x,dy)\\&\leq
\left(\frac{1}{2}\sup_{x\in\R}c_2(x)+\sup_{x\in\R}\int_{[-m,m]}y^{2}\nu_2(x,dy)\right)\xi^{2}+\int_{[-m,m]^{c}}(1-\cos\xi
y)\nu_1(x,dy),\end{align*} where in the fourth step we applied the
fact that $1-\cos y\leq y^{2}$ for all $y\in\R.$ Thus,
$$\sup_{x\in\R}q_2(x,\xi)\leq
c\xi^{2}+\sup_{x\in\R}\int_{[-m,m]^{c}}(1-\cos\xi y)\nu_1(x,dy),$$
where
$$c=\frac{1}{2}\sup_{x\in\R}c_2(x)+\sup_{x\in\R}\int_{[-m,m]}y^{2}\nu_2(x,dy).$$
By the same reasoning, we get
$$\sup_{x\in\R}\int_{[-m,m]^{c}}(1-\cos\xi
y)\nu_1(x,dy)\leq\sup_{x\in\R}q_1(x,\xi)\leq
\bar{c}\xi^{2}+\sup_{x\in\R}\int_{[-m,m]^{c}}(1-\cos\xi
y)\nu_1(x,dy),$$ for some $\bar{c}>0.$ Next, (\ref{eq:3.3}) implies
$$\lim_{\xi\longrightarrow0}\frac{\sup_{x\in\R}\int_{[-m,m]^{c}}(1-\cos\xi y)\nu_1(x,dy)}{\xi^{2}}=\infty,$$
$$\lim_{\xi\longrightarrow0}\frac{\sup_{x\in\R}q_1(x,\xi)}{\sup_{x\in\R}\int_{[-m,m]^{c}}(1-\cos\xi
y)\nu_1(x,dy)}=1\ \ \ $$ and
$$\ \ \ \ \ \lim_{\xi\longrightarrow0}\frac{c\xi^{2}+\sup_{x\in\R}\int_{[-m,m]^{c}}(1-\cos\xi
y)\nu_1(x,dy)}{\sup_{x\in\R}\int_{[-m,m]^{c}}(1-\cos\xi
y)\nu_1(x,dy)}=1.$$ Therefore, by applying Proposition \ref{p2.4},
the claim follows.

Now, we prove the transience case.  Again, let $m>0$ be so large
that $C\subseteq[-m,m]$. Clearly,
\begin{align*}\inf_{x\in\R}q_1(x,\xi)\geq\inf_{x\in\R}\int_{[-m,m]^{c}}(1-\cos\xi y)\nu_1(x,dy)\geq\inf_{x\in\R}\int_{[-m,m]^{c}}(1-\cos\xi y)\nu_2(x,dy)\end{align*}
and
\begin{align*}
q_2(x,\xi)&=\frac{1}{2}c_2(x)\xi^{2}+\int_{\R}(1-\cos\xi
y)\nu_2(x,dy)\\&=\frac{1}{2}c_2(x)\xi^{2}+\int_{[-m,m]}(1-\cos\xi
y)\nu_2(x,dy)+\int_{[-m,m]^{c}}(1-\cos\xi
y)\nu_2(x,dy)\\&\leq\frac{1}{2}\sup_{x\in\R}c_2(x)\xi^{2}+\xi^{2}\sup_{x\in\R}\int_{[-m,m]}y^{2}\nu_2(x,dy)+\sup_{x\in\R}\int_{[-m,m]^{c}}(1-\cos\xi
y)\nu_2(x,dy)\\&\leq
\xi^{2}\left(\frac{1}{2}\sup_{x\in\R}c_2(x)+\sup_{x\in\R}\int_{[-m,m]}y^{2}\nu_2(x,dy)\right)+\sup_{x\in\R}\int_{[-m,m]^{c}}(1-\cos\xi
y)\nu_2(x,dy).\end{align*} Thus,
$$\inf_{x\in\R}\int_{[-m,m]^{c}}(1-\cos\xi y)\nu_2(x,dy)\leq\inf_{x\in\R}q_2(x,\xi)\leq c\xi^{2}+\inf_{x\in\R}\int_{[-m,m]^{c}}(1-\cos\xi
y)\nu_2(x,dy),$$ where
$$c=\frac{1}{2}\sup_{x\in\R}c_2(x)+\sup_{x\in\R}\int_{[-m,m]}y^{2}\nu_2(x,dy).$$
Now, by (\ref{eq:3.2}), we get
$$\lim_{\xi\longrightarrow\infty}\frac{\inf_{x\in\R}\int_{[-m,m]^{c}}(1-\cos\xi
y)\nu_2(x,dy)}{\xi^{2}}=\infty,$$
$$\lim_{\xi\longrightarrow\infty}\frac{\inf_{x\in\R}q_2(x,\xi)}{\inf_{x\in\R}\int_{[-m,m]^{c}}(1-\cos\xi
y)\nu_2(x,dy)}=1\ \ \ $$ and
$$\ \ \ \ \ \lim_{\xi\longrightarrow\infty}\frac{c\xi^{2}+\inf_{x\in\R}\int_{[-m,m]^{c}}(1-\cos\xi
y)\nu_2(x,dy)}{\inf_{x\in\R}\int_{[-m,m]^{c}}(1-\cos\xi
y)\nu_2(x,dy)}=1.$$ Now, the desired result follows from Proposition
\ref{p2.4}.
\end{proof}
In many situations the Chung-Fuchs type conditions in
(\ref{eq:1.3}) and (\ref{eq:1.4})
 are
not operable. More precisely, it is not always easy to compute the
integrals appearing in (\ref{eq:1.3}) and (\ref{eq:1.4}). According
to this, in the sequel we derive  necessary and sufficient
conditions for the recurrence and transience of one-dimensional
symmetric nice Feller processes in terms of the L\'evy measure.
 First, recall that a symmetric Borel measure
$\mu(dy)$ on $\mathcal{B}(\R)$ is \emph{quasi-unimodal} if there
exists $y_0\geq0$ such that $y\longmapsto\mu(y,\infty)$ is a convex
function on $(y_0,\infty)$. Equivalently, a symmetric Borel measure
$\mu(dy)$ on $\mathcal{B}(\R)$ is quasi-unimodal if it is of the
form $\mu(dy)=\mu_0(dy)+f(y)dy,$ where the measure $\mu_0(dy)$ is
supported on $[-y_0,y_0]$, for some $y_0\geq0$, and the density
function $f(y)$ is supported on $[-y_0,y_0]^{c}$, it is symmetric
and decreasing on $(y_0,\infty)$ and
$\int_{y_0+\varepsilon}^{\infty}f(y)dy<\infty$ for every
$\varepsilon>0$ (see \cite[Chapters 5 and 7]{sato-book}). When
$y_0=0$, then $\mu(dy)$ is said to be \emph{unimodal}.
\begin{theorem}\label{tm3.5}
 Let $\process{F}$  be a
one-dimensional symmetric nice Feller process with  symbol
$q(x,\xi)$ and  L\'evy measure $\nu(x,dy)$. Assume that there exists
$x_0\in\R$ such that
\begin{enumerate}
\item [(i)] $\inf_{x\in\R}q(x,\xi)=q(x_0,\xi)$ for all $|\xi|$ small enough
              \item [(ii)] the L\'evy measure $\nu(x_0,dy)$ is
              quasi-unimodal
              \item [(iii)] there exists a
one-dimensional symmetric L\'evy process  $\process{L}$ with symbol
$q(\xi)$ and  L\'evy measure $\nu(dy)$, such that $\nu(x_0,dy)$ has
a bigger tail than $\nu(dy).$
 \end{enumerate}
Then, the transience property of $\process{L}$ implies
(\ref{eq:1.4}).
\end{theorem}
\begin{proof} By Theorem \ref{tm1.3}, it suffices to prove that
$$\int_{\{|\xi|<r\}}\frac{d\xi}{\inf_{x\in\R}q(x,\xi)}=\int_{\{|\xi|<r\}}\frac{d\xi}{q(x_0,\xi)}<\infty$$
for some $r>0$. Let $\process{F^{0}}$ be a L\'evy process with
symbol $q(x_0,\xi)$. Now, by \cite[Theorem 38.2]{sato-book},
$\process{F^{0}}$ is transient. Hence, by \cite[Theorem
37.5]{sato-book},
$$\int_{\{|\xi|<r\}}\frac{d\xi}{q(x_0,\xi)}<\infty$$ for all $r>0.$
\end{proof}
As a direct consequence of the above result we get the following
corollary.
\begin{corollary}\label{c3.6}
 Let $\process{F}$  be a
one-dimensional symmetric nice Feller process with  symbol
$q(x,\xi)$ and  L\'evy measure $\nu(x,dy)$. Assume that there exists
$x_0\in\R$ such that
\begin{enumerate}
\item [(i)] $\sup_{x\in\R}q(x,\xi)=q(x_0,\xi)$ for all $|\xi|$ small enough
              \item [(ii)] there exists a
one-dimensional symmetric L\'evy process  $\process{L}$ with symbol
$q(\xi)$ and  L\'evy measure $\nu(dy)$, such that $\nu(dy)$ is
quasi-unimodal and has a bigger tail than $\nu(x_0,dy).$
 \end{enumerate}
Then, the recurrence property of $\process{L}$ implies
(\ref{eq:1.3}).
\end{corollary}
For the necessity of the quasi-unimodality assumption in Theorem
\ref{tm3.5} and Corollary \ref{c3.6} see
 \cite[Theorem 38.4]{sato-book}. Explicit examples of  one-dimensional symmetric nice Feller
processes which satisfy the conditions in Theorem \ref{tm3.5} and
Corollary \ref{c3.6} can be easily constructed in the classes of
stable-like processes and  Feller processes  obtained by variable
order subordination (see Section \ref{s2}).

\begin{theorem} \label{tm3.7}Let $\process{F}$  be a
one-dimensional symmetric nice Feller process with  symbol
$$q(x,\xi)=\frac{1}{2}c(x)\xi^{2}+2\int_0^{\infty}(1-\cos\xi
y)\nu(x,dy).$$ Let us define
$$R\left(x,r,y\right):=\nu\left(x,\bigcup_{n=0}^{\infty}(2nr+y,2(n+1)r-y]\right),$$
for $x\in\R$ and $r\geq y\geq0.$ Then, for arbitrary $\rho>0$,
\begin{align}\label{eq:3.6}\int_\rho^{\infty}\left(\sup_{x\in\R}\int_0^{r}yR(x,r,y)dy\right)^{-1}dr=\infty\end{align}
if, and only if, (\ref{eq:1.3}) holds true. Further,  under
 (\ref{eq:3.2}), (\ref{eq:1.4}) holds true if, and only if,
\begin{align}\label{eq:3.7}\int_\rho^{\infty}\left(\inf_{x\in\R}\int_0^{r}yR(x,r,y)dy\right)^{-1}dr<\infty\end{align}
holds for all $\rho>0$ large enough.
\end{theorem}
\begin{proof} We follow the proof of \cite[Theorem 38.3]{sato-book}.
Let us denote $N(x,y):=\nu(x,(y,\infty)),$ for $x\in\R$ and
$y\geq0$. Then, we have
\begin{align*}q(x,\xi)-\frac{1}{2}c(x)\xi^{2}&=2\int_0^{\infty}(1-\cos\xi
y)\nu(x,dy)\\&=2\int_0^{\infty}(1-\cos\xi
y)d(-N(x,y))\\&=2\xi\int_0^{\infty}N(x,y)\sin\xi
y\,dy\\&=2\xi\sum_{n=0}^{\infty}\int_0^{\frac{2\pi}{\xi}}N\left(x,\frac{2\pi
n}{\xi}+y\right)\sin\xi
y\,dy\\&=2\xi\sum_{n=0}^{\infty}\left(I_{n,1}+I_{n,2}+I_{n,3}+I_{n,4}\right),\end{align*}
where in the third step we applied the integration by parts formula
and in the final step we wrote
\begin{align*}I_{n,1}&=\int_0^{\frac{\pi}{2\xi}}N\left(x,\frac{2\pi
n}{\xi}+y\right)\sin\xi y\,dy,\\
I_{n,2}&=\int_{\frac{\pi}{2\xi}}^{\frac{\pi}{\xi}}N\left(x,\frac{2\pi
n}{\xi}+y\right)\sin\xi
y\,dy=\int_{0}^{\frac{\pi}{2\xi}}N\left(x,\frac{2\pi
n}{\xi}+\frac{\pi}{\xi}-y\right)\sin\xi y\,dy,\\
I_{n,3}&=\int_{\frac{\pi}{\xi}}^{\frac{3\pi}{2\xi}}N\left(x,\frac{2\pi
n}{\xi}+y\right)\sin\xi
y\,dy=-\int_{0}^{\frac{\pi}{2\xi}}N\left(x,\frac{2\pi
n}{\xi}+\frac{\pi}{\xi}+y\right)\sin\xi y\,dy\end{align*} and
$$\ \ I_{n,4}=\int_{\frac{3\pi}{2\xi}}^{\frac{2\pi}{\xi}}N\left(x,\frac{2\pi
n}{\xi}+y\right)\sin\xi
y\,dy=-\int_{0}^{\frac{\pi}{2\xi}}N\left(x,\frac{2\pi
n}{\xi}+\frac{2\pi}{\xi}-y\right)\sin\xi y\,dy.$$ Thus,
$$I_{n,1}+I_{n,4}=\int_0^{\frac{\pi}{2\xi}}\nu\left(x,\left(\frac{2\pi
n}{\xi}+y,\frac{2\pi(n+1)}{\xi}-y\right]\right)\sin\xi y\,dy\ \ \ \
\ \ \ \ $$ and
$$I_{n,2}+I_{n,3}=\int_0^{\frac{\pi}{2\xi}}\nu\left(x,\left(\frac{\pi
(2n+1)}{\xi}-y,\frac{\pi(2n+1)}{\xi}+y\right]\right)\sin\xi y\,dy.$$
Now, by defining
$$\bar{R}(x,r,y):=\nu\left(x,\bigcup_{n=0}^{\infty}((2n+1)r-y,(2n+1)r+y]\right),$$
we have
$$q(x,\xi)-\frac{1}{2}c(x)\xi^{2}=2\xi\left(\int_0^{\frac{\pi}{2\xi}}R\left(x,\frac{\pi}{\xi},y\right)\sin\xi y\,dy+\int_0^{\frac{\pi}{2\xi}}\bar{R}\left(x,\frac{\pi}{\xi},y\right)\sin\xi
y\,dy\right).$$ Further, note that
$$R\left(x,\frac{\pi}{\xi},y\right)\geq
\bar{R}\left(x,\frac{\pi}{\xi},y\right)\geq0,\quad
y\in\left(0,\frac{\pi}{2\xi}\right],$$ and $$\frac{2y}{\pi}\leq\sin
y\leq y,\quad y\in\left(0,\frac{\pi}{2}\right].$$ This yields
\begin{align}\label{eq:3.8}\frac{4}{\pi}\xi^{2}\int_0^{\frac{\pi}{2\xi}}yR\left(x,\frac{\pi}{\xi},y\right)dy\leq
q(x,\xi)-
\frac{1}{2}c(x)\xi^{2}\leq4\xi^{2}\int_0^{\frac{\pi}{2\xi}}yR\left(x,\frac{\pi}{\xi},y\right)dy.\end{align}
Next,  we have
\begin{align*}&\int_0^{\frac{\pi}{\rho}}\left(\xi^{2}\sup_{x\in\R}\int_0^{\frac{\pi}{\xi}}yR\left(x,\frac{\pi}{\xi},y\right)dy\right)^{-1}d\xi=\frac{1}{\pi}\int_\rho^{\infty}\left(\sup_{x\in\R}\int_0^{r}yR\left(x,r,y\right)dy\right)^{-1}dr\end{align*} and
\begin{align*}&\int_0^{\frac{\pi}{\rho}}\left(\xi^{2}\inf_{x\in\R}\int_0^{\frac{\pi}{\xi}}yR\left(x,\frac{\pi}{\xi},y\right)dy\right)^{-1}d\xi=\frac{1}{\pi}\int_\rho^{\infty}\left(\inf_{x\in\R}\int_0^{r}yR\left(x,r,y\right)dy\right)^{-1}dr,\end{align*}
where we made the substitution $\xi\longmapsto\pi/r$. Thus,
(\ref{eq:3.6}) implies
$$\int_0^{\frac{\pi}{\rho}}\frac{d\xi}{\sup_{x\in\R}\left(q(x,\xi)-\frac{1}{2}c(x)\xi^{2}\right)}=\infty$$ and
$$\int_0^{\frac{\pi}{\rho}}\frac{d\xi}{\inf_{x\in\R}\left(q(x,\xi)-\frac{1}{2}c(x)\xi^{2}\right)}<\infty$$ implies (\ref{eq:3.7}).
Finally,  from (\ref{eq:3.2}) and (\ref{eq:3.3}), we have
\be\label{eq:3.9}\lim_{\xi\longrightarrow0}\frac{\sup_{x\in\R}q(x,\xi)}{\sup_{x\in\R}\left(q(x,\xi)-\frac{1}{2}c(x)\xi^{2}\right)}=\lim_{\xi\longrightarrow0}\frac{\inf_{x\in\R}q(x,\xi)}{\inf_{x\in\R}\left(q(x,\xi)-\frac{1}{2}c(x)\xi^{2}\right)}=1.\ee
Now, the claim follows from Proposition \ref{p2.4}.

To prove the converse, first note that
\begin{align*}&\int_{0}^{\frac{\pi}{\xi}}y\nu\left(x,\left(\frac{2n\pi}{\xi}+y,\frac{2(n+1)\pi}{\xi}-y\right]\right)dy\\&=
\int_{0}^{\frac{\pi}{2\xi}}y\nu\left(x,\left(\frac{2n\pi}{\xi}+y,\frac{2(n+1)\pi}{\xi}-y\right]\right)dy+\int_{\frac{\pi}{2\xi}}^{\frac{\pi}{\xi}}y\nu\left(x,\left(\frac{2n\pi}{\xi}+y,\frac{2(n+1)\pi}{\xi}-y\right]\right)dy\\&=
\int_{0}^{\frac{\pi}{2\xi}}y\nu\left(x,\left(\frac{2n\pi}{\xi}+y,\frac{2(n+1)\pi}{\xi}-y\right]\right)dy+4\int_{\frac{\pi}{4\xi}}^{\frac{\pi}{2\xi}}y\nu\left(x,\left(\frac{2n\pi}{\xi}+2y,\frac{2(n+1)\pi}{\xi}-2y\right]\right)dy\\&\leq
\int_{0}^{\frac{\pi}{2\xi}}y\nu\left(x,\left(\frac{2n\pi}{\xi}+y,\frac{2(n+1)\pi}{\xi}-y\right]\right)dy+4\int_{\frac{\pi}{4\xi}}^{\frac{\pi}{2\xi}}y\nu\left(x,\left(\frac{2n\pi}{\xi}+y,\frac{2(n+1)\pi}{\xi}-y\right]\right)dy\\&\leq
5\int_{0}^{\frac{\pi}{2\xi}}y\nu\left(x,\left(\frac{2n\pi}{\xi}+y,\frac{2(n+1)\pi}{\xi}-y\right]\right)dy.\end{align*}
Hence,
$$\int_{0}^{\frac{\pi}{\xi}}yR\left(x,\frac{\pi}{\xi},y\right)dy\leq5\int_{0}^{\frac{\pi}{2\xi}}yR\left(x,\frac{\pi}{\xi},y\right)dy,$$ that is,
\begin{align*}
\int_\rho^{\infty}\left(\sup_{x\in\R}\int_0^{r}yR\left(x,r,y\right)dy\right)^{-1}dr
&=\int_0^{\frac{\pi}{\rho}}\left(\xi^{2}\sup_{x\in\R}\int_0^{\frac{\pi}{\xi}}yR\left(x,\frac{\pi}{\xi},y\right)dy\right)^{-1}d\xi\\
&\geq\frac{1}{5}\int_0^{\frac{\pi}{\rho}}\left(\xi^{2}\sup_{x\in\R}\int_0^{\frac{\pi}{2\xi}}yR\left(x,\frac{\pi}{\xi},y\right)dy\right)^{-1}d\xi\end{align*}
and
\begin{align*}
\int_\rho^{\infty}\left(\inf_{x\in\R}\int_0^{r}yR\left(x,r,y\right)dy\right)^{-1}dr
&=\int_0^{\frac{\pi}{\rho}}\left(\xi^{2}\inf_{x\in\R}\int_0^{\frac{\pi}{\xi}}yR\left(x,\frac{\pi}{\xi},y\right)dy\right)^{-1}d\xi\\
&\geq\frac{1}{5}\int_0^{\frac{\pi}{\rho}}\left(\xi^{2}\inf_{x\in\R}\int_0^{\frac{\pi}{2\xi}}yR\left(x,\frac{\pi}{\xi},y\right)dy\right)^{-1}d\xi,\end{align*}
where in the first steps we applied the substitution
$\xi\longmapsto\pi/r$. Thus, (\ref{eq:1.3}) and (\ref{eq:3.7}), by
using (\ref{eq:3.2}), (\ref{eq:3.3}), (\ref{eq:3.8}), (\ref{eq:3.9})
and Proposition \ref{p2.4}, imply (\ref{eq:3.6}) and (\ref{eq:1.4}),
respectively.
\end{proof}

As a consequence of  Theorem \ref{tm3.7}, we also get the following
characterization of the recurrence and transience  in terms of the
tail behavior of the L\'evy measure.
\begin{corollary}\label{c3.8} Let $\process{F}$  be a
one-dimensional symmetric nice Feller process with  symbol
$q(x,\xi)$ and  L\'evy measure $\nu(x,dy)$. Let us define
$$N\left(x,y\right):=\nu\left(x,(y,\infty)\right),$$
for $x\in\R$ and $y\geq0.$  Then, for  arbitrary $\rho>0$,
\begin{align}\label{eq:3.10}\int_\rho^{\infty}\left(\sup_{x\in\R}\int_0^{r}yN(x,y)dy\right)^{-1}dr=\infty\end{align}
implies (\ref{eq:1.3}), and (\ref{eq:1.4}) implies
\begin{align}\label{eq:3.11}\int_\rho^{\infty}\left(\inf_{x\in\R}\int_0^{r}yN(x,y)dy\right)^{-1}dr<\infty\end{align} for  all $\rho>0$ large
enough.
\end{corollary}
\begin{proof} The claim directly follows from the fact
$N(x,y)\geq R(x,r,y)$ for all $x\in\R$ and all $0\leq y\leq r$.
\end{proof}
In addition, if we assume the quasi-unimodality of the L\'evy
measure $\nu(x,dy)$, then we can prove the equivalence in Corollary
\ref{c3.8}.
\begin{theorem}\label{tm3.9}Let $\process{F}$  be a
one-dimensional symmetric nice Feller process with symbol $q(x,\xi)$
and  L\'evy measure $\nu(x,dy)$, such that
\begin{enumerate}
\item [(i)]$\nu(x,dy)$ is quasi-unimodal uniformly in $x\in\R$
              \item [(ii)] the function $x\longmapsto\nu(x,O-x)$
            is lower
              semicontinuous for every open set $O\subseteq\R$, that
              is, $\liminf_{y\longrightarrow
              x}\nu(y,O-y)\geq\nu(x,O-x)$ for all $x\in\R$ and all open sets $O\subseteq\R$.
 \end{enumerate}  Then,
(\ref{eq:3.10}) holds true if, and only if, (\ref{eq:1.3}) holds
true. Further, if (\ref{eq:3.2})  and
\be\label{eq:3.12}\inf_{x\in\R}\int_{y_0}^{\infty}y\nu(x,(y,\infty))dy>0\ee
hold true for some $y_0>0$, then (\ref{eq:3.11}) holds true if, and
only if, (\ref{eq:1.4}) holds true.
\end{theorem}
\begin{proof}The proof is divided  in three steps.

\textbf{Step 1.} In the first step, we construct a nice Feller
processes $\process{\bar{F}}$ with finite and unimodal L\'evy
measure  which has the same tails as $\nu(x,dy)$. Then, in
particular, by Theorem \ref{tm3.1},  the conditions in
(\ref{eq:1.3}) and (\ref{eq:1.4}) are equivalent for
$\process{\bar{F}}$ and $\process{F}$.
 By  assumptions
(i) and (\textbf{C2}),   there exists $y_0>1$ such
that $\sup_{x\in\R}\nu(x,(y_0-1,\infty))<\infty$ and
$y\longmapsto\nu(x,(y,\infty))$ is convex on $(y_0-1,\infty)$ for
all $x\in\R$. Next, let $f(x,y)$ be the density of $\nu(x,dy)$ on
$(y_0-1,\infty)$, that is, $\frac{\partial}{\partial
y}\nu(x,(y,\infty))=-f(x,y)$ on $(y_0-1,\infty)$.
                                                    Further, note
                                                    that
                                                    $\sup_{x\in\R}f(x,y_0)<\infty.$
                                                    Indeed, if this was not
                                                    the case, then
                                                    we would have
                                                    $$\sup_{x\in\R}\nu(x,(y_0-1,\infty))=\sup_{x\in\R}\int_{y_0-1}^{\infty}f(x,y)dy\geq\sup_{x\in\R}f(x,y_0)\int_{y_0-1}^{y_0}dy=\sup_{x\in\R}f(x,y_0)=\infty,$$  where in the second step we employed the fact that $y\longmapsto f(x,y)$ is decreasing on $(y_0-1,\infty)$ for all $x\in\R$.
Let $c:=y_0\sup_{x\in\R}f(x,y_0)+\sup_{x\in\R}\nu(x,(y_0,\infty))+1$
and let  $\bar{\nu}(x,dy)$ be a symmetric probability kernel on
$(\R,\mathcal{B}(\R))$ given by $\bar{\nu}(x,\{0\})=0$ and
$$\bar{\nu}(x,(y,\infty)):=\left\{\begin{array}{cc}
                                                     \displaystyle{\frac{\nu(x,(y_0,\infty))-c}{2cy_0}\lambda(0,y)+\frac{1}{2}}, & 0<y\leq y_0 \\
                                                      \displaystyle{\frac{\nu(x,(y,\infty))}{2c}},&
                                                      y\geq y_0,
                                                    \end{array}\right.$$
                                                    for all
                                                    $x\in\R.$
Note that, since $$\frac{\partial}{\partial
y}\bar{\nu}(x,(y,\infty))\big|_{y_0}=-\frac{f(x,y_0)}{2c}\quad
\textrm{and}\quad \frac{\partial}{\partial
y}\bar{\nu}(x,(y,\infty))=\frac{\nu(x,(y_0,\infty))-c}{2cy_0}$$ for
$y\in(0,y_0)$,
                                                    $\bar{\nu}(x,dy)$ is
                                                    unimodal.
Next, put $\bar{p}(x,dy):=\bar{\nu}(x,dy-x)$.  Clearly,
 $\bar{p}(x,dy)$ is a probability kernel on
$(\R,\mathcal{B}(\R))$ which defines a Markov chain, say
$\chain{F}$. Further, let $\process{P}$ be the Poisson process with
intensity $\lambda=2c$ independent of $\chain{F}$. Then, by
$\bar{F}_t:=F_{P_t}$, $t\geq0$, is  well defined  a Markov process
with the transition kernel
$$\mathbb{P}^{x}(\bar{F}_t\in dy)=e^{-2c
               t}\sum_{n=0}^{\infty}\frac{
               (2ct)^{n}}{n!}\bar{p}^{n}\left(x,dy\right),$$
               here $\bar{p}^{0}\left(x,dy\right)$ is the Dirac
               measure $\delta_x(dy)$ and
               $$\bar{p}^{n}\left(x,dy\right):=\int_{\R}\ldots\int_{\R}\bar{p}\left(x,dy_1\right)\ldots\bar{p}\left(y_{n-1},dy\right),$$
               for $n\geq1.$
Note that $\process{\bar{F}}$ is a nice Feller process with  symbol
$\bar{q}(x,\xi)=2c\int_\R(1-\cos\xi y)\bar{\nu}(x,dy)$. Indeed, the
               strong continuity property can be easily verified.
               Next,  in order to prove
               the continuity property of $x\longmapsto\int_\R\mathbb{P}^{x}(\bar{F}_t\in dy)f(y)$, for $f\in C_b(\R)$ and $t\geq0$, by \cite[Proposition
               6.1.1]{meyn-tweedie-book}, it suffices to show the
               lower semicontinuity property of the function
               $x\longmapsto\bar{\nu}(x,O-x)$  for all open
               sets $O\subseteq\R$. But this is the
               assumption (ii). Finally, we show that the
               function $x\longmapsto\int_\R\mathbb{P}^{x}(\bar{F}_t\in dy)f(y)$
               vanishes
               at infinity for all $f\in C_\infty(\R)$ and all $t\geq0$. Let $f\in C_\infty(\R)$ and $\varepsilon>0$ be
arbitrary and let $m>0$ be such that  $||f||_\infty\leq m$. Since
$C_c(\R)$ is dense in $(C_\infty(\R),||\cdot||_\infty)$, there
exists $f_\varepsilon\in C_c(\R)$ such that
$||f-f_\varepsilon||_\infty<\varepsilon$. We have
\begin{align*}\left|\int_\R \bar{p}(x,dy)f(y)\right|&\leq\int_\R
\bar{p}(x,dy)|f(y)|\\&<\int_\R
\bar{p}(x,dy)|f_\varepsilon(y)|+\varepsilon\\&=\int_{\R}\bar{\nu}(x,dy)|
f_\varepsilon(y+x)|dy+\varepsilon\\&\leq
(m+\varepsilon)\int_{\rm{supp}\,\it{f}_\varepsilon-x}\bar{\nu}(x,dy)+\varepsilon\\&=(m+\varepsilon)\bar{\nu}(x,\rm{supp}\,\it{f}_\varepsilon-x)+\varepsilon.\end{align*}
Now, since
$\rm{supp}\,\it{f}_\varepsilon:=\{y:f_{\varepsilon}(y)\neq\rm{0}\}$
has compact closure, it suffices to prove that
$\lim_{|x|\longrightarrow\infty}\bar{\nu}(x,C-x)=0$ for every
compact set $C\subseteq\R.$ Let $C\subseteq\R$ be a compact set.
Then, for arbitrary $r>0$ and $|x|$ large enough, we have
\begin{align*}\bar{\nu}(x,C-x)=\frac{\nu(x,C-x)}{2c}\leq\frac{\nu(x,(-r,r)^{c})}{
                                      2c}\leq\frac{\sup_{x\in\R}\nu(x,(-r,r)^{c})}{
                                      2c}.\end{align*}
                                      Hence,
$$\limsup_{|x|\longrightarrow\infty}\bar{\nu}(x,C-x)\leq\frac{\sup_{x\in\R}\nu(x,(-r,r)^{c})}{2c}.$$
                                      Now, by letting
                                      $r\longrightarrow\infty$,
from \cite[Theorem 4.4]{rene-conserv}, we get the desired result.
Finally, it can be easily verified that the symbol of
$\process{\bar{F}}$ is given by $\bar{q}(x,\xi)=2c\int_\R(1-\cos\xi
y)\bar{\nu}(x,dy)$  and obviously, by the definition,
$\process{\bar{F}}$  satisfies conditions
(\textbf{C1})-(\textbf{C4}).

\textbf{Step 2.} In Corollary \ref{c3.8} we have proved that
(\ref{eq:3.10}) implies (\ref{eq:1.3}). In the second step, we prove
the converse. Since $\bar{\nu}(x,dy)$ is unimodal, by \cite[Exercise
29.21]{sato-book}, there exists a random variable $F_x$ such that
$\bar{\nu}(x,dy)$ is the distribution of the random variable $UF_x$,
where $U$ is  uniformly distributed random variable on $[0,1]$
independent of $F_x$. Further, let $\bar{\nu}_U(x,dy)$ be the
distribution of the random variable $F_x$. By \cite[Lemma
38.6]{sato-book},
$\bar{\nu}_U(x,(y,\infty))\geq\bar{\nu}(x,(y,\infty))$ for all
$x\in\R$ and all $y\geq0$.
  Now, we have
\begin{align*}\bar{q}(x,\xi)&=2c\int_\R(1-\cos\xi y)\bar{\nu}(x,dy)\\&=2c\int_0^{1}\int_{\R}(1-\cos(\xi uy))\bar{\nu}_U(x,dy)du\\&=2c\int_\R\left(1-\frac{\sin\xi y}{\xi y}\right)\bar{\nu}_U(x,dy).\end{align*}
 Further,    since $$1-\frac{\sin y}{y}\geq \bar{c}\min\{1, y^{2}\}$$ for
 all $y\in\R$ and all $0<\bar{c}<\frac{1}{6}$,
              we have $$\bar{q}(x,\xi)\geq2c\bar{c}\int_\R\min\{1,(\xi y)^{2}\}\bar{\nu}_U(x,dy)=8c\bar{c}\xi^{2}\int_0^{\frac{1}{|\xi|}}y\bar{N}_U(x,y)dy,$$
where $\bar{N}_U(x,y):=\bar{\nu}_U(x,(y,\infty)),$ for $x\in\R$ and $y\geq0$. Finally, let us put
$\bar{N}(x,y):=\bar{\nu}(x,(y,\infty))$, for $x\in\R$ and $y\geq0$,
then we have
\begin{align*}\int_{\rho}^{\infty}\left(\sup_{x\in\R}\int_0^{r}y\bar{N}(x,y)dy\right)^{-1}dr&\geq\int_{\rho}^{\infty}\left(\sup_{x\in\R}\int_0^{r}y\bar{N}_U(x,y)dy\right)^{-1}dr\\&=\int_{\{|\xi|<\frac{1}{\rho}\}}\left(\xi^{2}\sup_{x\in\R}\int_0^{\frac{1}{|\xi|}}y\bar{N}_U(x,y)dy\right)^{-1}d\xi\\&\geq8c\bar{c}\int_{\{|\xi|<\frac{1}{\rho}\}}\frac{d\xi}{\sup_{x\in\R}\bar{q}(x,\xi)}.\end{align*}
Further, we have
\begin{align*}&\lim_{r\longrightarrow\infty}\frac{\sup_{x\in\R}\int_{y_0}^{r}yN(x,y)dy}{\sup_{x\in\R}\int_0^{y_0}y\bar{N}(x,y)dy+\frac{1}{2c}\sup_{x\in\R}\int_{y_0}^{r}yN(x,y)dy}\\&\leq\liminf_{r\longrightarrow\infty}\frac{\sup_{x\in\R}\int_0^{r}yN(x,y)dy}{\sup_{x\in\R}\int_0^{r}y\bar{N}(x,y)dy}\\&\leq\limsup_{r\longrightarrow\infty}\frac{\sup_{x\in\R}\int_0^{r}yN(x,y)dy}{\sup_{x\in\R}\int_0^{r}y\bar{N}(x,y)dy}\\&\leq\lim_{r\longrightarrow\infty}\frac{\sup_{x\in\R}\int_0^{y_0}yN(x,y)dy+\sup_{x\in\R}\int_{y_0}^{r}yN(x,y)dy}{\frac{1}{2c}\sup_{x\in\R}\int_{y_0}^{r}yN(x,y)dy}\end{align*}
Now, if $\sup_{x\in\R}\int_{y_0}^{\infty}yN(x,y)dy=0$ the claim
trivially follows. On the other hand, if\\
$\sup_{x\in\R}\int_{y_0}^{\infty}yN(x,y)dy>0$, we have
\be\label{eq:3.13}0<\liminf_{r\longrightarrow\infty}\frac{\sup_{x\in\R}\int_0^{r}yN(x,y)dy}{\sup_{x\in\R}\int_0^{r}y\bar{N}(x,y)dy}\leq\limsup_{r\longrightarrow\infty}\frac{\sup_{x\in\R}\int_0^{r}yN(x,y)dy}{\sup_{x\in\R}\int_0^{r}y\bar{N}(x,y)dy}<\infty,\ee
   which together with Proposition \ref{p2.4}  and Theorem \ref{tm3.1} proves the claim.

 \textbf{Step 3.} In the third step,  we consider the second part of the theorem. In Corollary \ref{c3.8} we have proved that
 (\ref{eq:1.4}) implies (\ref{eq:3.11}). Now, to prove the converse,
by completely the same arguments as in the second step, we have
\begin{align*}8c\bar{c}\int_{\{|\xi|<\frac{1}{\rho}\}}\frac{d\xi}{\inf_{x\in\R}\bar{q}(x,\xi)}\leq\int_{\rho}^{\infty}\left(\inf_{x\in\R}\int_0^{r}y\bar{N}(x,y)dy\right)^{-1}dr.\end{align*}
Now, the claim follows from Proposition \ref{p2.4} (ii),
(\ref{eq:3.2}), Theorem \ref{tm3.1}, (\ref{eq:3.12}) and a similar
argumentation as in (\ref{eq:3.13}).
\end{proof}
 In
the sequel we discuss some consequences of Theorem \ref{tm3.9}.
\begin{theorem}\label{tm3.10}Let $\process{F}$  be a
one-dimensional symmetric nice Feller process with  symbol
$$q(x,\xi)=\frac{1}{2}c(x)\xi^{2}+2\int_0^{\infty}(1-\cos\xi
y)\nu(x,dy)$$ satisfying  (\ref{eq:3.2}). Then, for any $r>0$, and
any $0<a<b<\pi$ fixed, either one of the following conditions
\begin{align}\label{eq:3.14}\int_{r}^{\infty}\frac{dy}{y^{2}\inf_{x\in\R}\left(N\left(x,by\right)-N\left(x,(a+\pi)y\right)\right)}<\infty,\end{align}
or
\begin{align}\label{eq:3.15}\int_r^{\infty}\frac{d\xi}{\inf_{x\in\R}\int_0^{\xi}y^{2}\nu(x,dy)}<\infty\end{align}
  implies (\ref{eq:1.4}).
So that if (\ref{eq:3.14}) or (\ref{eq:3.15}) is satisfied, then the
process $\process{F}$ is transient.
\end{theorem}
\begin{proof} Let $N(x,y):=\nu(x,(y,\infty)),$ for $x\in\R$ and $y\geq0$. Now, by
the integration by parts formula, we have
\begin{align*}q(x,\xi)-\frac{1}{2}c(x)\xi^{2}&=2\int_0^{\infty}(1-\cos\xi
y)\nu(x,dy)\\&=2\int_0^{\infty}(1-\cos\xi
y)d(-N(x,y))\\
&=2\xi\int_0^{\infty}\sin\xi y\,N(x,y)dy\\&=2\int_0^{\infty}\sin
y\,N\left(x,\frac{y}{\xi}\right)dy\\&\geq 2\int_0^{2\pi}\sin
y\,N\left(x,\frac{y}{\xi}\right)dy\\&=2\int_0^{\pi}\sin
y\left(N\left(x,\frac{y}{\xi}\right)-N\left(x,\frac{y+\pi}{\xi}\right)\right)dy,\end{align*}
where in the last two inequalities we used the periodicity of the
sine function  and the nonincreasing property of $y\longmapsto
N(x,y)$. Therefore, for any $0<a<b<\pi$, we have that
$$q(x,\xi)-\frac{1}{2}c(x)\xi^{2}\geq2\int_a^{b}\sin
y\left(N\left(x,\frac{y}{\xi}\right)-N\left(x,\frac{y+\pi}{\xi}\right)\right)dy,$$
and hence $$q(x,\xi)-\frac{1}{2}c(x)\xi^{2}\geq
c(a,b)\left(N\left(x,\frac{b}{\xi}\right)-N\left(x,\frac{a+\pi}{\xi}\right)\right),$$
where $c(a,b):=2(b-a)\inf_{y\in(a,b)}\sin y.$ This yields
\begin{align*}\int_{\{|\xi|<\frac{1}{r}\}}\frac{d\xi}{\inf_{x\in\R}\left(q(x,\xi)-\frac{1}{2}c(x)\xi^{2}\right)}&\leq c^{-1}(a,b)\int_{0}^{\frac{1}{r}}\frac{d\xi}{\inf_{x\in\R}\left(N\left(x,\frac{b}{\xi}\right)-N\left(x,\frac{a+\pi}{\xi}\right)\right)}\\&
=c^{-1}(a,b)\int_{r}^{\infty}\frac{dy}{y^{2}\inf_{x\in\R}\left(N\left(x,by\right)-N\left(x,(a+\pi)y\right)\right)},\end{align*}
which together with Proposition \ref{p2.4}  and Theorem \ref{tm3.1}
proves the desired result.

To prove the second claim, first note that
$$q(x,\xi)-\frac{1}{2}c(x)\xi^{2}=2\int_0^{\infty}(1-\cos\xi
y)\nu(x,dy)\geq\frac{2}{\pi}\xi^{2}\int_0^{\frac{1}{|\xi|}}y^{2}\nu(x,dy),$$
where  we applied  the fact that $1-\cos y\geq\frac{1}{\pi}y^{2}$
for all $|y|\leq\frac{\pi}{2}.$ This yields
\begin{align*}\int_{\{|\xi|<\frac{1}{r}\}}\frac{d\xi}{\inf_{x\in\R}\left(q(x,\xi)-\frac{1}{2}c(x)\xi^{2}\right)}&\leq\frac{\pi}{2}\int_{\{|\xi|<\frac{1}{r}\}}\frac{d\xi}{\xi^{2}\inf_{x\in\R}\int_0^{\frac{1}{|\xi|}}y^{2}\nu(x,dy)}\\&=\pi\int_{r}^{\infty}\frac{d\xi}{\inf_{x\in\R}\int_0^{\xi}y^{2}\nu(x,dy)},\end{align*}
where in the second step we made the substitution
$\xi\longrightarrow\xi^{-1}$. Now, the desired result is an
immediate consequence of Proposition \ref{p2.4} and Theorem
\ref{tm3.1}.

\end{proof}
In addition, by assuming the quasi-unimodality property of the L\'evy measure we get the following sufficient condition for transience.

\begin{theorem}\label{tm3.11}Let $\process{F}$  be a
one-dimensional symmetric nice Feller process with  symbol
$q(x,\xi)$ and  L\'evy measure $\nu(x,dy)$ satisfying the
assumptions from Theorem \ref{tm3.9}. Further, let $y_0>0$  be a
constant of uniform quasi-unimodality  of $\nu(x,dy)$. Then,
$$\int_{y_0}^{\infty}\frac{dy}{y^{3}\inf_{x\in\R}f(x,y)}<\infty$$
implies (\ref{eq:1.4}).
\end{theorem}
\begin{proof}  By Theorem \ref{tm3.9}, it suffices to show that $$\int_{y_0}^{\infty}\left(\inf_{x\in\R}\int_0^{r}yN(x,y)dy\right)^{-1}dr<\infty.$$ For all $r\geq 2y_0$, we have
\begin{align*}\int_0^{r}yN(x,y))dy&\geq\int_{y_0}^{r}\int_{y}^{\infty}yf(x,u)dudy\\&=\int_{y_0}^{r}\int_y^{r}yf(x,u)dudy+\int_{y_0}^{r}\int_r^{\infty}yf(x,u)dudy\\
&\geq\int_{y_0}^{r}\int_y^{r}yf(x,u)dudy\\&=
\frac{r^{3}-3ry_0^{2}+2y_0^{3}}{6}f(x,r)\\ &\geq \frac{2y_0^{3}}{3}r^{3}f(x,r).
\end{align*}  Note that in
the fourth inequality we took into account the fact that the
densities $f(x,y)$ are decreasing on $(y_0,\infty)$ for all
$x\in\R$. Now, we have
$$\int_{2y_0}^{\infty}\left(\inf_{x\in\R}\int_0^{r}yN(x,y)dy\right)^{-1}dr\leq\frac{3}{2y_0^{3}}\int_{2y_0}^{\infty}\frac{dy}{y^{3}\inf_{x\in\R}f(x,y)}\leq\frac{3}{2y_0^{3}}\int_{y_0}^{\infty}\frac{dy}{y^{3}\inf_{x\in\R}f(x,y)},$$
which completes the proof.
\end{proof}
Note that Theorem \ref{tm3.11} can be strengthened. By assuming only
uniform quasi-unimodality (and the condition in (\ref{eq:3.2})) of
the L\'evy measure $\nu(x,dy)$,  we have
\begin{align*}\int_{r}^{\infty}\frac{dy}{\inf_{x\in\R}\int_0^{y}u^{2}\nu(x,du)}&\leq\int_{r}^{\infty}\frac{dy}{\inf_{x\in\R}\int_{y_0}^{y}u^{2}f(x,u)du}\\&\leq3\int_{r}^{\infty}\frac{d\xi}{(y^{3}-y^{3}_0)\inf_{x\in\R}f(x,y)}\\&\leq
c\int_{r}^{\infty}\frac{d\xi}{y^{3}\inf_{x\in\R}f(x,y)},\end{align*}
where $y_0>0$ is a constant of uniform quasi-unimodality of
$\nu(x,dy)$ and $r>y_0$ and $c>\frac{3r^{3}}{r^{3}-y_0^{3}}$ are
arbitrary. Now, the claim is a direct consequence of Theorem
\ref{tm3.10}. Let us also remark that in the L\'evy process case the
condition from Theorem \ref{tm3.11} holds true even without the
quasi-unimodality assumption, that is, the corresponding density
does not have to be decreasing (see \cite{sandric-levy}).

We conclude this paper with some comparison conditions for the recurrence and transience in terms of the L\'evy measure.
Directly from Theorem \ref{tm3.9} we can generalize the results from
Theorem \ref{tm3.5} and Corollary \ref{c3.6}.
\begin{theorem}\label{tm3.12}
Let $\process{F^{1}}$ and $\process{F^{2}}$  be  one-dimensional
symmetric nice Feller processes with  symbols $q_1(x,\xi)$ and
$q_2(x,\xi)$ and  L\'evy measures $\nu_1(x,dy)$ and and
$\nu_2(x,dy)$, respectively. Let us put
$$N_1\left(x,y\right):=\nu_1\left(x,(y,\infty)\right)\quad \textrm{and}\quad N_2\left(x,y\right):=\nu_2\left(x,(y,\infty)\right),$$
for $x\in\R$ and $y\geq0$. If $N_1\left(x,y\right)$ has a bigger
tail than $N_2\left(x,y\right)$, uniformly in $x\in\R$, then, for
arbitrary $\rho>0$,
$$\int_\rho^{\infty}\left(\sup_{x\in\R}\int_0^{r}yN_1(x,y)dy\right)^{-1}dr=\infty$$
implies
$$\int_\rho^{\infty}\left(\sup_{x\in\R}\int_0^{r}yN_2(x,y)dy\right)^{-1}dr=\infty.$$
 In addition, if  (\ref{eq:3.12}) holds true, then, for arbitrary
$\rho>0$,
$$\int_\rho^{\infty}\left(\inf_{x\in\R}\int_0^{r}yN_2(x,y)dy\right)^{-1}dr<\infty$$
 implies
$$\int_\rho^{\infty}\left(\inf_{x\in\R}\int_0^{r}yN_1(x,y)dy\right)^{-1}dr<\infty.$$
\end{theorem}
\begin{theorem}  \label{tm3.13}
Let $\process{F^{1}}$ and  $\process{F^{2}}$ be  one-dimensional
symmetric nice Feller processes with  symbols $q_1(x,\xi)$ and
$q_2(x,\xi)$ and  L\'evy measures $\nu_1(x,dy)$ and $\nu_2(x,dy)$,
respectively. Further, assume  that
\begin{enumerate}
\item [(i)]$\nu_1(x,dy)$ is quasi-unimodal uniformly in $x\in\R$
\item [(ii)] $\nu_1(x,dy)$ has a bigger tail than $\nu_2(x,dy)$ uniformly in $x\in\R$
              \item [(iii)] the function $x\longmapsto\nu_1(x,O-x)$
            is lower
              semicontinuous for every open set $O\subseteq\R$.
 \end{enumerate}   Then,  for arbitrary $r>0$,
$$\int_{\{|\xi|<r\}}\frac{d\xi}{\sup_{x\in\R}q_1(x,\xi)}=\infty$$
implies
$$\int_{\{|\xi|<r\}}\frac{d\xi}{\sup_{x\in\R}q_2(x,\xi)}=\infty.$$
In addition, if $q_1(x,\xi)$ satisfies  (\ref{eq:3.2}) and
(\ref{eq:3.12}), then, for all $r>0$ small enough,
$$\int_{\{|\xi|<r\}}\frac{d\xi}{\inf_{x\in\R}q_2(x,\xi)}<\infty$$
implies
$$\int_{\{|\xi|<r\}}\frac{d\xi}{\inf_{x\in\R}q_1(x,\xi)}<\infty.$$

\end{theorem}
Finally, let us remark that, according to Proposition \ref{p2.4}, if
the function
$\xi\longmapsto\inf_{x\in\R}\sqrt{\it{q}\left(x,\xi\right)}$ is
subadditive, then  the statements of Theorems \ref{tm3.7},
\ref{tm3.9} and \ref{tm3.13} and Corollary  \ref{c3.8} (involving
the conditions in (\ref{eq:1.4}) and (\ref{eq:3.11})) do not depend
on $r>0$ and $\rho>0.$
 \section*{Acknowledgement} This work has been  supported in part by Croatian Science Foundation under the project $3526$. The author would like to thank   Bj\"{o}rn B\"{o}ttcher
 for turning his attention to Theorem \ref{tm2.6}.  Many thanks to Zoran Vondra\v{c}ek and Jian Wang for
 helpful suggestions and comments.  The author
also thanks the anonymous reviewer for careful reading of the paper
and for helpful comments that led to improvement of the
presentation.

\bibliographystyle{plain}
\bibliography{References}

\end{document}